\documentclass[a4paper]{article}
\usepackage{amsmath, amsthm, amssymb, amstext, color, amsfonts, epsf}

\theoremstyle{plain}
\newtheorem{thm}{Theorem}
\newtheorem{lem}[thm]{Lemma}

\newtheorem{cor}[thm]{Corollary}

\theoremstyle{definition}

\newcommand{\C}{\ensuremath{\mathcal{C}}}

\newcommand{\N}{\ensuremath{\mathbb{N}}}
\newcommand{\D}{\ensuremath{\mathcal{D}}}

\newcommand{\T}{\ensuremath{\mathcal{T}}}
\newcommand{\V}{\ensuremath{\mathcal{V}}}
\newcommand{\W}{\ensuremath{\mathcal{W}}}

\newcommand{\e}{\ensuremath{\varepsilon}}
\renewcommand{\P}{\ensuremath{\mathcal{P}}}

\newcommand{\Q}{\ensuremath{\mathcal{Q}}}
\newcommand{\Hc}{\ensuremath{\mathcal{H}}}

\newcommand{\COMMENT}[1]{}

\newcommand{\?}[1]{%
  \marginpar{%
    \begin{minipage}{\marginparwidth}\small%
      \begin{flushleft}%
        #1%
      \end{flushleft}%
   \end{minipage}%
  }%
}%

\renewcommand{\?}[1]{}

\newcommand{\boundary}{\partial}
\newcommand{\closure}{\overline}
\newcommand{\interior}{\mathring}

\renewcommand{\hom}{\simeq}

\usepackage{enumerate}

\usepackage{graphicx}

\title{On the excluded minor structure theorem\\ for graphs of large treewidth}

\author{Reinhard Diestel, Ken-ichi Kawarabayashi, \\ Theodor M\"uller, Paul Wollan}
 \date{21 November, 2011}

\begin{document}  
    
\maketitle

\begin{abstract}

At the core of the Robertson-Seymour theory of graph minors lies
a powerful structure theorem which captures, for any fixed graph~$H$, the common structural features of all the graphs not containing $H$ as a minor.\penalty-200 Robertson and Seymour prove several versions of this  theorem, each stressing some particular aspects needed at a corresponding stage of the proof of the main result of their theory, the graph minor theorem.

We prove a new version of this structure theorem: one that seeks to combine maximum applicability with a minimum of technical ado, and which might serve as a canonical version for future applications in the broader field of graph minor theory.  Our proof departs from a simpler version proved explicitly by Robertson and Seymour. It then uses a combination of traditional methods and new techniques to derive some of the more subtle features of other versions as well as further useful properties, with substantially simplified proofs.
\end{abstract}


\section{Introduction}

Graphs in this paper are finite and may have loops and multiple edges. Otherwise we use the terminology of~\cite{diestel}.
A graph $H$ is a \emph{minor} of a graph $G$ if $H$ can be obtained
from a subgraph of $G$ by contracting edges.

The theory of graph minors was developed by Robertson and Seymour, in a series of 23 papers published over more than twenty years, with the aim of proving a single result: the {\em graph minor theorem\/}, which says that in any infinite collection of finite graphs there
 is one that is a minor of another. As with other deep results in mathematics, the body of theory developed for the proof of the graph minor theorem has also found applications elsewhere, both within graph theory and computer science. Yet many of these applications rely not only on the general techniques developed by Robertson and Seymour to handle graph minors, but also on one particular auxiliary result that is also central to the proof of the graph minor theorem: a~result describing the structure of all graphs $G$ not containing some fixed other graph~$H$ as a minor.

This structure theorem has many facets. It roughly says that
every graph $G$ as above can be decomposed into parts that can each be `almost' embedded in a surface of bounded genus (the bound depending on~$H$ only), and which fit together in a tree structure~\cite[Thm.~12.4.11]{diestel}. Although later dubbed a `red herring' (in the search for the proof of the graph minor theorm) by Robertson and Seymor themselves~\cite{GM16}, this simplest version of the structure theorem is the one that appears now to be best known, and which has also found the most algorithmic applications \cite{demaine1,demaine2,demaine3,kmstoc}.

A particularly simple form of this structure theorem applies when the excluded minor $H$ is planar: in that case, the said parts of~$G$---the parts that fit together in a tree-structure and together make up all of~$G$---have bounded size, i.e., $G$~has bounded {\em tree-width\/}. If $H$ is not planar, the graphs $G$ not containing $H$ as a minor have unbounded tree-width, and therefore contain arbitrarily large grids as minors and arbitrarily large walls as topological minors~\cite{diestel}. Such a large grid or wall identifies, for every low-order separation of~$G$, one side in which most of that grid or wall lies. This is formalized by the notion of a {\em tangle\/}: the larger the tree-width of~$G$, the larger the grid or wall, the order of the separations for which this works, and (thus) the {\em order\/} of the tangle. Since adjacent parts in our tree-decomposition of $G$ meet in only a bounded number of vertices
   and thus define low-order separations, our large-order tangle `points to' one of the parts, the part $G'$ that contains most of its defining grid or wall.

The more subtle versions of the structure theorem, such as Theorem~(13.4) from Graph Minors~XVII~\cite{GM17}, now focus just on this part $G'$ of~$G$. Like every part in our decomposition, it intersects every other part in a controlled way. Every such intersection consists of a bounded number of vertices, of which some lie in a fixed {\em apex set\/} $A\subseteq V(G')$ of bounded size, while the others are either at most 3 vertices lying on a face boundary of the portion $G_0$ of~$G'$ embedded in the surface, or else lie in (a common bag of) a so-called {\em vortex\/}, a ring-like subgraph of $G'$ that is not embedded in the surface and meets $G_0$ only in (possibly many) vertices of a face boundary of~$G_0$. The precise structure of these vortices, of which $G'$ has only boundedly many, will be the focus of our attention for much of the paper. Our theorem describes in detail both the inner structure of the vortices and the way in which they are linked to each other and to the large wall, by disjoint paths in the surface. These are the properties that have been used in applications of the structure theorems such as~\cite{BKMM,DKW}, and which will doubtless be important also in future applications. An important part of the proofs is a new technique for analyzing vortices.  We note that these techniques have also been independently developed by Geelen and Huynh~\cite{GH}.

The basis for this paper is Theorem~(3.1) from Graph Minors~XVI~\cite{GM16}, which we shall restate as Theorem~\ref{thm:3.1}. Together with the `grid-theorem' that large enough tree-width forces arbitrarily large grid minors (see~\cite{diestel}), and a simple fact about tangles from Graph Minors~X~\cite{GM10}, these are all the results we require from the Graph Minor series.

This paper is organized as follows. In Section~\ref{sec:structure_thms} we introduce the terminology we need to state our results, as well as the theorem from Graph Minors~XVI~\cite{GM16} on which we shall base our proof. Section~\ref{sec:findingtd} explains how we can find the tree-decomposition indicated earlier, with some additional information on how the parts of the tree-decomposition overlap. Section~\ref{sec:tools} collects some lemmas about graphs embedded in a surface, partly from the literature and partly new. In Section~\ref{sec:modifying} we show how a given near-embedding of a graph can be simplified in various ways if we allow ourselves to remove a bounded number of vertices (which, in applications of these tools, will be added to the apex set). Section~\ref{sec:pathsystems} contains lemmas showing how to obtain path systems with nice properties. Section~\ref{sec:mainresult} contains the proof of our structure theorem. In the last section, we give an alternative definition of vortex decompositions and show that our result works with these `circular' decompositions as well.


\section{Structure Theorems}
\label{sec:structure_thms}


A \textit{vortex} is a pair $V = (G,\Omega)$, where $G$ is a graph and $\Omega =: \Omega(V)$ is a linearly ordered set $(w_1, \ldots, w_n)$ of vertices in~$G$. These vertices are the \textit{society vertices} of the vortex; their number $n$ is its \textit{length}. We do not always distinguish notationally between a vortex and its underlying graph or the vertex set of that graph; for example, a \textit{subgraph of $V$} is just a subgraph of~$G$, a \textit{subset of $V$} is a subset of $V(G)$, and so on. Also, we will often use $\Omega$ to refer to the linear order of the vertices $w_1, \dots, w_n$ as well as the set of vertices $\{w_1, \dots, w_n\}$.

A path-decomposition $\D = (X_1,\ldots,X_m)$ of $G$ is a \textit{decomposition} of our vortex~$V$ if $m=n$ and $w_i\in X_i$ for all~$i$. The \textit{depth} of the vortex~$V$ is the minimum width of a path-decomposition of $G$ that is a decomposition of~$V$.

When $n>1$, the \emph{adhesion} of our decomposition $\D$ of $V$ is the maximum value of ${|X_{i}\cap X_{i+1}|}$, taken over all $1\leq i < n$. We define the \textit{adhesion} of a vortex $V$ as the minimum adhesion of a decomposition of that vortex.

When $\D$ is a decomposition of a vortex $V$ as above, we write ${Z_i:=(X_{i}\cap X_{i+1})\setminus \Omega}$, for all $1 \leq i< n$.  These $Z_i$ are the \emph{adhesion sets} of $\D$. We call $\D$ \textit{linked} if 
\begin{itemize}
\item all these $Z_i$ have the same size;
\item there are $|Z_i|$ disjoint $Z_{i-1}$--$Z_{i}$ paths in~$G[X_i]-\Omega$, for all $1<i<n$;
\item $X_i \cap \Omega =\{w_{i-1},w_{i}\}$ for all $1\leq i \leq n$, where $w_0:=w_1$.
\end{itemize}
Note that $X_i\cap X_{i+1} = Z_i \cup \{w_i\}$, for all $1\le i < n$ (Fig.~\ref{fig:linkedvortex}).

\begin{figure}[ht]
\centering
\noindent
\includegraphics{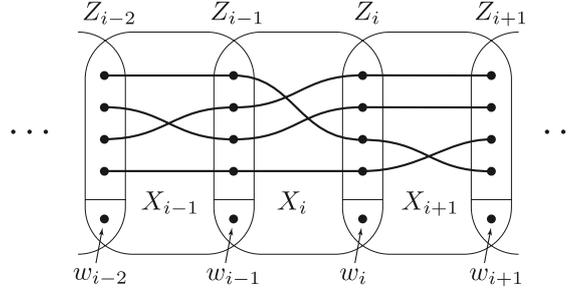}
\caption{A linked vortex decomposition}\label{fig:linkedvortex}
\end{figure}

The union over all $1<i<n$ of the $Z_{i-1}$--$Z_{i}$ paths in a linked decomposition of $V$ is a disjoint union of $X_1$--$X_n$ paths in~$G$; we call the set of these paths a \emph{linkage} of~$V$ with respect to $(X_1,\ldots,X_m)$.

Clearly, if $V$ has a linked decomposition as above, then $G$ has no edges between non-consecutive society vertices, since none of the $X_i$ could contain both ends of such an edge. Conversely, if $G$ has no such edges then $V$ does have a linked decomposition: just let $X_i$ consist of all the vertices of $G-\Omega$ plus $w_{i-1}$ and~$w_{i}$. We shall be interested in linked vortex decompositions whose adhesion is small, unlike  in this example. 


Let $V=(G,\Omega)$ be a vortex, and $v$ a vertex of some supergraph of $V$. Clearly, $(G-v,\Omega\setminus\{v\})$ is a vortex, too, which we denote by $V-v$\?{$V-v$}. If the length of $V$ is greater than 2,
this operation cannot increase the adhesion $q$ of $V$: This is clear for $v\notin\Omega$, so suppose $\Omega=(w_1,\ldots,w_n)$ with $v=w_k$ for some $1\leq k \leq n$. We may assume without loss of generality that $k\neq n$. Take a decomposition $(X_1,\ldots,X_n)$ of $V$ of adhesion $q$. Then, it is easy to see that 
 \[
 	(X_1,\ldots, X_{k-1},(X_{k}\cup X_{k+1})\setminus\{w_k\},X_{k+1},\ldots,X_n)
\] 
is a decomposition of $V-v$ of adhesion at most $q$. We shall not be interested in the adhesion of vortices of length at most 2. For a vertex set $A\subseteq V$ we denote by $V-A$\?{$V-A$} the vortex we obtain by deleting the vertices of $A$ in turn. For a set of vortices $\V$ we define $\V-A:=\{V-A:V\in \V, V-A\neq \emptyset\}$.\?{$\V-A$}

A (directed) \textit{separation} of a graph $G$ is an ordered pair $(A,B)$ of non-empty subsets of $V(G)$ such that $G[A]\cup G[B] = G$. The number $|A\cap B|$ is the \textit{order} of $(A,B)$. Whenever we speak of separations in this paper, we shall mean such directed separations.

A set $\T$\?{$\T$} of separations of~$G$, all of order less than some integer~$\theta$, is a \textit{tangle of order $\theta$} if the following holds: 
\begin{enumerate}[(1)]
\item For every separation $(A,B)$ of $G$ of order less than $\theta$, either $(A,B)$ or $(B,A)$ lies in~$\T$.
\item If $(A_i,B_i)\in \T$ for $i=1,2,3$, then $G[A_1]\cup G[A_2] \cup G[A_3] \neq G$.
\end{enumerate}

Note that if $(A,B)\in\T$  then $(B,A)\notin\T$; we think of $A$ as the `small side' of the separation~$(A,B)$, with respect to this tangle.

Given a tangle $\T$ of order $\theta$ in a graph $G$, and a set $Z\subseteq V(G)$ of fewer than $\theta$ vertices, let $\T-Z$\?{$\T-Z$} denote the set of all separations  $(A',B')$ of $G-Z$ of order less than $\theta-|Z|$ such that there exists a separation $(A,B)\in \T$ with $Z\subseteq A\cap B$, $A-Z=A'$ and $B-Z=B'$. It is shown in \cite[Theorem (6.2)]{GM10} that $\T-Z$ is a tangle of order $\theta-|Z|$ in $G-Z$.

Given a subset $D$ of a surface $\Sigma$, we write $\interior D$, $\boundary D$, and $\closure D$ for the topological interior, boundary, and closure, of $D$ in $\Sigma$, respectively. 
For positive integers $\alpha_0, \alpha_1, \alpha_2$ and $\alpha:=(\alpha_0,\alpha_1,\alpha_2)$, a graph $G$ is \textit{$\alpha$-nearly embeddable} in $\Sigma$ if there is a subset $A\subseteq V(G)$ with $|A|\leq \alpha_0$ such that there are integers  $\alpha'\leq \alpha_1$ and $n\geq \alpha'$ for which $G-A$ can be written as the union of $n+1$ edge-disjoint graphs $G_0,\ldots,G_n$ with the following properties:
\begin{enumerate}[(i)]
	\item  For all  $1\leq i\leq j \leq n$ and $\Omega_i:= V(G_i\cap G_0)$, the pairs $(G_i,\Omega_i)=:V_i$ are vortices, and $G_i\cap G_j \subseteq G_0$ when $i\neq j$ . 
	\item The vortices $V_1,\ldots,V_{\alpha'}$ are disjoint and have adhesion at most $\alpha_2$; we denote the set of these vortices by $\V$. 
	We will sometimes refer to these vortices as \emph{large} vortices.
	\item The vortices $V_{\alpha'+1},\ldots,V_n$ have length at most 3; we denote the set of these vortices by $\W$. 
	These are the \emph{small} vortices of the near-embedding.
	\item There are closed discs in~$\Sigma$, with disjoint interiors $D_1,\ldots, D_n$, and an embedding ${\sigma: G_0 \hookrightarrow \Sigma-\bigcup_{i=1}^n D_i}$ such that $\sigma(G_0)\cap\boundary D_i = \sigma(\Omega_i)$ for all~$i$ and the generic linear ordering of $\Omega_i$ is compatible with the natural cyclic ordering of its image (i.e., coincides with the linear ordering of $\sigma(\Omega_i)$ induced by $[0,1)$ when $\boundary D_i$ is viewed as a suitable homeomorphic copy of $[0,1]/\{0,1\}$). For $i=1,\dots,n$ we think of the disc $D_i$ as \emph{accommodating} the (unembedded) vortex~$V_i$, and denote $D_i$ as~$D(V_i)$.
\end{enumerate}

We call $(\sigma,G_0,A,\V,\W)$ an \emph{$\alpha$-near embedding} of $G$ in~$\Sigma$, or just a \emph{near-embedding}, with \emph{apex set} $A$. For an integer $\alpha'$ larger than all the $\alpha_i$ we call $(\sigma,G_0,A,\V,\W)$ an \emph{$\alpha'$-near embedding}. It \textit{captures} a tangle $\T$ of $G$ if the `large side' $B'$ of an element $(A',B')\in \T-A$ is never contained in a vortex. 

A direct implication of Theorem~(3.1) from \cite{GM16}, stated in this terminology, reads as follows: 
\begin{thm}\label{thm:3.1}
For every non-planar graph $R$ there exist integers $\theta,\alpha \geq 0$ such that the following holds: Let $G$ be a graph that does not contain $R$ as a minor, and let $\T$ be a tangle in $G$ of order at least $\theta$. Then $G$ has an $\alpha$-near embedding, with apex set $A$ say, in a surface $\Sigma$ in which $R$ cannot be drawn, and this embedding captures $\T-A$.
\end{thm}

We shall use Theorem~\ref{thm:3.1} as the basis of our proofs in this paper.

Given a near-embedding $(\sigma,G_0,A,\V,\W)$ of $G$, let $G'_0$ be the graph resulting from $G_0$ by joining any two nonadjacent vertices $u,v\in G_0$ that lie in a common small vortex $V\in\W$; the new edge $uv$ of $G'_0$ will be called a \emph{virtual edge}. By embedding these virtual edges disjointly in the discs $D(V)$ accommodating their vortex~$V$, we extend our embedding $\sigma\colon G_0\hookrightarrow\Sigma$ to an embedding $\sigma'\colon G'_0\hookrightarrow\Sigma$. We shall not normally distinguish $G'_0$ from its image in $\Sigma$ under~$\sigma'$.

A vortex $(G_i,\Omega_i)$ is \emph{properly attached} if $|\Omega_i|\leq 3$ and it satisfies the following two requirements. First, for every pair of distinct vertices $u,v\in \Omega_i$ the graph $G_i$ must contain an $\Omega_i$-path (one with no inner vertices in~$\Omega_i$) from $u$ to~$v$. Second, whenever $u,v,w\in\Omega_i$ are distinct vertices (not necessarily in this order), there are two internally disjoint $\Omega_i$-paths in $G_i$ linking $u$ to~$v$ and $v$ to~$w$, respectively.

Clearly, if $(G_i,\Omega_i)$ is properly attached to $G_0$, the vortex $(G_i-v,\Omega_i\setminus\{v\})$ is properly attached to $G_0-v$ for any vertex $v\in\Omega_i$. 

Given a graph $H$ embedded in our surface $\Sigma$, a curve $C$ in $\Sigma$ is \emph{$H$-normal} if it hits $H$ in vertices only. The \textit{distance in $\Sigma$} of two points $x,y\in \Sigma$ is the minimal value of $|G'_0\cap C|$ taken over all $G'_0$-normal curves $C$ in the surface that link $x$ to~$y$.
The \textit{distance in $\Sigma$} of two vortices $V$ and $W$  is the minimum distance in $\Sigma$ of a vertex in $\Omega(V)$ from a vertex  in~$\Omega(W)$. Similar, the  \textit{distance in $\Sigma$} of two subgraphs $H$ and $H'$ of $G'_0$ is the minimum distance in $\Sigma$ of a vertex in $H$ from a vertex in~$H'$. 

A cycle $C$ in $\Sigma$ is \emph{flat} if $C$ bounds an open disc $D(C)$ in $\Sigma$. Disjoint cycles $C_1,\ldots,C_n$ in $\Sigma$ are \emph{concentric} if they bound open discs $D(C_1) \supseteq \ldots \supseteq D(C_n)$ in~$\Sigma$. A set $\P$ of paths intersects $C_1,\ldots,C_n$ \emph{orthogonally}, and is \emph{orthogonal to $C_1,\ldots,C_n$}, if every path $P$ in $\P$ intersects each of the cycles in a (possibly trivial but non-empty) subpath of~$P$.


Let $G$ be a graph embedded in a surface $\Sigma$, and $\Omega$ a subset of its vertices. Let $C_1,\ldots,C_n$ be cycles in $G$ that are concentric in $\Sigma$. The cycles $C_1,\ldots,C_n$ \emph{enclose} $\Omega$ if $\Omega \subseteq D(C_n)$. They \emph{tightly enclose} $\Omega$ if, in addition, the following holds:
\begin{equation*}
	\begin{minipage}[c]{0.8\textwidth}\em
		For all $1\leq k \leq n$ and every point $v\in \boundary D(C_k)$, there is a vertex $w\in\Omega$ 
		whose distance from $v$ in $\Sigma$ is at most $n-k+2$.
	\end{minipage}\ignorespacesafterend 
\end{equation*}

For a near-embedding $(\sigma,G_0,A,\V,\W)$ of a graph $G$ in a surface $\Sigma$ and concentric cycles $C_1,\ldots,C_n$ in $G'_0$, a vortex $V\in \V$ is \emph{(tightly) enclosed} by these cycles if they (tightly) enclose $\Omega(V)$.

A flat triangle in $G'_0$ is a \emph{boundary triangle} if it bounds a disc that is a face of $G'_0$ in~$\Sigma$.

For positive integers~$r\geq 3$, define a graph $H_r$ as follows (Fig.~\ref{fig:H6}).  Let $P_1, \dots, P_r$ 
be $r$ disjoint (`horizontal') paths of length $r-1$, say $P_i = v_1^i\dots
v_{r}^i$. Let $V(H_r) = \bigcup_{i=1}^r V(P_i)$, and let
\begin{equation*}
	\begin{split}
		E(H_r) = \bigcup_{i=1}^r E(P_i) \cup
		\Big\{&v_j^i v_j^{i+1} \mid \text{ $i,j$ odd};\ 1 \le i < r;\ 1 \le j \le r\Big\} \\
		& \cup \Big\{v_j^i v_j^{i+1} \mid \text{ $i,j$ even};\ 1 \le i < r;\ 1 \le j \le r\Big\}.
	\end{split}
\end{equation*}

We call the paths $P_i$ the \emph{rows} of $H_r$; the paths induced by the vertices $\{v^i_j,v^i_{j+1}:1\leq i\leq r\}$ for an odd index $i$ are its \emph{columns}.
\begin{figure}[ht]
\centering
\noindent
\includegraphics{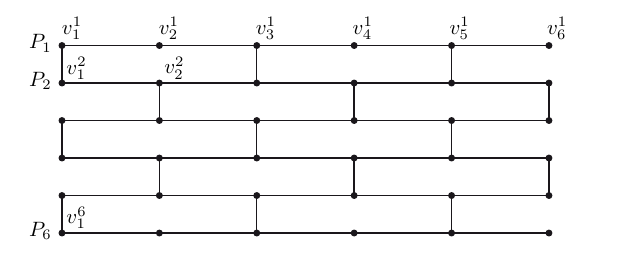}
\caption{The graph $H_6$}\label{fig:H6}
\end{figure}

The 6-cycles in $H_r$ are its \emph{bricks}. In the natural plane embedding of~$H_r$, these bound faces of $H$. The outer cycle of the unique maximal 2-connected subgraph of $H_r$ is the \emph{boundary cycle}  of $H_r$.

Any subdivision $H = T H_r$ of $H_r$ will be called an \emph{$r$-wall}, or a \emph{wall of size $r$}. The \emph{bricks} and the \emph{boundary cycle} of $H$ are its subgraphs that form subdivisions of the bricks and the boundary cycle of~$H_r$, respectively. An embedding of $H$ in a surface~$\Sigma$ is a \emph{flat} embedding, and $H$ is \emph{flat} in~$\Sigma$, if the boundary cycle $C$ of $H$ bounds an open disc $D(H)$ in $\Sigma$ such that all its bricks $B_i$ bound disjoint, open discs $D(B_i)$ in $\Sigma$ with $D(B_i)\subseteq D(H)$ for all $i$.

For topological concepts used but not defined in this paper we refer to~\cite[Appendix~B]{diestel}. When we speak of the {\em genus\/} of a surface $\Sigma$ we always mean its Euler genus, the number $2-\chi(\Sigma)$.

A closed curve $C$ in $\Sigma$ is \emph{genus-reducing} if the (one or two) surfaces obtained by `capping the holes' of the components of $\Sigma\setminus C$ have smaller genus than $\Sigma$. Note that if $C$ separates $\Sigma$ and one of the two resulting surfaces is homeomorphic to $S^2$, the other is homeomorphic to~$\Sigma$. Hence in this case $C$ was not genus-reducing.

The \emph{representativity} of an embedding $G\hookrightarrow\Sigma\not\simeq S^2$ is the smallest integer $k$ such that every genus-reducing curve $C$ in $\Sigma$ that meets $G$ only in vertices meets it in at least $k$ vertices. We remark that, by \cite[Lemmas~B.5 and~B.6]{diestel}, all faces of an embedded graph are discs if the representativity of the embedding is positive.

An $(\alpha_0, \alpha_1, \alpha_2)$-near embedding $(\sigma,G_0,A,\V,\W)$ of a graph $G$ in some surface $\Sigma$ is {\it ($\beta$, r)-rich} for integers $3\leq \beta \leq r$ if the following statements hold:
\begin{enumerate}[(i)]
	\item  \label{prop:bigwall} $G'_0$ contains a flat $r$-wall $H$.
	\item If $\Sigma \not \hom S^2$, the representativity of $G'_0$ in $\Sigma$ is at least $\beta$.
	\item For every vortex $V\in \V$ there are $\beta$ concentric cycles $C_1(V),\ldots,C_{\beta}(V)$ in $G'_0$ tightly enclosing $V$ and bounding open discs $D_1(V) \supseteq \ldots \supseteq D_\beta(V)$, such that $D_\beta(V)$ contains $\Omega(V)$ and $\closure{D(H)}$ does not meet $\closure{D_1(V)}$. 
	For distinct, large vortices ${V,W\in\V}$, the discs $\closure{D_1(V)}$ and $\closure{D_1(W)}$ are disjoint. In particular, every two vortices in $\V$ have distance greater than $\beta$ in $\Sigma$.
	\item \label{prop:linked} Let $V \in \V$ with $\Omega(V)=(w_1,\ldots,w_n)$. Then there is a linked decomposition of $V$ of adhesion at most $\alpha_2$ and a path $P$ in $V\cup \bigcup \W$ with $V(P\cap G_0)=\Omega(V)$ that avoids all the paths of the linkage of~$V$, and traverses $w_1,\ldots,w_n$ in this order.
	\item \label{prop:linkagetogrid} For every vortex $V\in\V$, its set of society vertices $\Omega(V)$ is linked in $G'_0$ to branch vertices of $H$ by a set $\P(V)$ of $\beta$ disjoint paths having no inner vertices in $H$.
	\item \label{prop:orthogonalpaths} For every vortex $V\in\V$, the paths in $\P(V)$ intersect the cycles $C_1(V),\ldots,C_\beta(V)$ orthogonally.
	\item All vortices in $\W$ are properly attached.
\end{enumerate}\bigbreak

Using this terminology, we can now state the main result of our paper:
\begin{thm}\label{thm:richthm}
For every non-planar graph~$R$ and integers~$3\leq \beta\leq r$ there exist integers $\alpha_0 = \alpha_0(R,\beta)$, $\alpha_1 = \alpha_1(R)$  and $w = w(\alpha_0, R,\beta, r)$ such that the following holds with $\alpha=(\alpha_0,\alpha_1,\alpha_1)$.
Every graph~$G$ of tree-width~$tw(G)\geq w$ that does not contain $R$ as a minor has an $\alpha$-near, $(\beta, r)$-rich embedding in some surface~$\Sigma$ in which $R$ cannot be embedded.
\end{thm}

For our proof of Theorem~\ref{thm:richthm} we shall use Theorem~\ref{thm:3.1}, but not directly. Instead, we use Theorem~\ref{thm:3.1} in the next section to prove Theorem~\ref{thm:extended1.3}, stated below, which is a strengthening of  Theorem~(1.3) of~\cite{GM16}. Our proof of Theorem~\ref{thm:richthm} will then be based on Theorem~\ref{thm:extended1.3}.


\section{Finding a tree-decomposition}
\label{sec:findingtd}


The following lemma shows that we can slightly modify a given $\alpha$-near embedding by embedding some more vertices of the graph in the surface, so that all the small vortices are properly attached to $G_0$. 

\begin{lem}\label{lem:properlyattached}
Given an integer $\alpha$ and  an $\alpha$-near embedding $(\sigma,G_0,A,\V,\W)$ of a graph $G$ in a surface $\Sigma$, there exists an $\alpha$-near embedding $(\hat\sigma,\hat G_0,A,\V,\hat\W)$ of $G$ in $\Sigma$ such that $G_0\subseteq \hat G_0$ and $\hat\sigma|_{G_0}=\sigma$, each vortex in $\hat\W$ is properly attached to $\hat G_0$, and consecutive society vertices of vortices $V\in\V$ are never adjacent in $V$.
\end{lem}

\begin{proof}
Let us consider the following modifications of our  near-embedding, each resulting in another $\alpha$-near embedding.

\begin{enumerate}
	\item[(1)] 
		By embedding edges between society vertices of a small vortex $V$ in $D(V)$,
		we may assume that no vortex in $\W$ contains an edge between
		two of its society vertices.
	\item[(2)]
		By first performing (1) and then splitting a small vortex $(G_i,\Omega_i)$ into several small vortices 
		each consisting of a component of $G_i-\Omega_i$ together with only their neighbours in $\Omega_i$ 
		as society vertices, we may assume that every $(G_i,\Omega_i)\in\W$ satisfies
		\begin{equation}
			\begin{minipage}[c]{0.8\textwidth}\em
				The graph $G_i-\Omega_i$ is connected and receives an edge from every vertex in~$\Omega_i$.
			\end{minipage}\tag{$\star$}\ignorespacesafterend 
		\end{equation}
	\item[(3)]
		If a society vertex $w$ of a small vortex $V=(G_i,\Omega_i)$ has only one neighbour $v$ in $G_i$, 
		we can embed $v$ and all $v-\Omega_i$ edges in $D(V)$ and replace $w$ with $v$ in $\Omega_i$. 
		Thus, we may assume that for every vortex $(G_i, \Omega_i)\in\W$ every society vertex has at least two
		neighbours in $G_i-\Omega_i$.
	\item[(4)] 
		Let $V:=(G_i,\Omega_i)\in \W$ be a vortex of length 3. 
		If there is a vertex $z\in V(G_i)\setminus\Omega_i$, 
		that separates one society vertex $w\in\Omega_i$ from the other two society vertices 
		$w',w''$, we can write $(G_i,\Omega_i)$ as the union of 
		two small vortices $V^1:=(G_i^1,\{z,w\})$ and $V^2:=(G_i^2,\{z,w',w''\})$. 
		Let $G_0^+$ denote the graph we obtain from $G_0$ by adding $z$ to its vertex set and
		extending $\sigma$ to an embedding $\sigma^+$ of $G_0^+$ in $\Sigma$ by mapping $z$ to a point in $D(V)$.
		It is easy to see that  $(\sigma^+,G_0^+, A, \V,(\W\setminus\{V\})\cup\{V^1,V^2\})$ 
		is an $\alpha$-near embedding of $G$ in $\Sigma$.
\end{enumerate}
 
We can iterate these two modifications only finitely often: Every application of (1), (3) or (4) 
increases either the number of embedded vertices or the number of embedded edges of the graph~$G$ 
while an application of (2) reduces the number of small vortices not satisfying $(\star)$.

Let $(\hat\sigma,\hat G_0,A,\V,\hat\W)$ be the $\alpha$-near embedding obtained by applying 
the two modifications as often as possible. Then, for every vortex $(G_i,\Omega_i)\in\W$ 
every $w\in\Omega_i$ has at least two neighbours in $G_i-\Omega_i$, which is connected. 
In particular every two vertices in $\Omega_i$ are linked by an $\Omega_i$-path in $G_i$.

Suppose now that $\Omega_i=\{u,v,w\}$, and let us find paths $P=u\ldots v$ and $Q=v\ldots w$ in $G_i$ that meet only in $v$.
Let $v',v''$ be distinct neighbours of $v$ in $G_i-\Omega_i$. We can find $P$ and $Q$ as desired 
unless the sets $\{v',v''\}$ and $\{u,w\}$ are separated in $G_i$ by one vertex $z$. 
Then $z\neq v$, since $G_i-\Omega_i$ is connected and $u,w$ send edges there.
So $z$ also separates $v$ from $\{u, w\}$ in $G_i$, contrary to $(4)$.

Thus, $(G_i, \Omega_i)$ is properly attached. Clearly, $G_0\subseteq \hat G_0$ and $\hat \sigma |_{G_0} = \sigma$. Embedding any vortex edges between adjacent society vertices of large vortices $V$ in the surface instead, we may assume that $V$ contains no such edges, as desired.
\end{proof}

Given two graphs $G$ and $H$, we say that $H$ is \emph{properly attached to $G$} if the vortex $(H,V(H)\cap V(G))$ is properly attached to $G$.


\begin{thm}\label{thm:extended1.3}
For every non-planar graph $R$ and for every integer $m$ there exists an integer $\alpha$ such that for every graph $G$ that does not contain $R$ as a minor and every $Z\subseteq V(G)$ with $|Z|\leq m$ there exist a tree-decomposition $(V_t)_{t\in T}$ of $G$ and a choice $r\in V(T)$ of a \emph{root} of $T$ such that, for every $t\in T$, there is a surface $\Sigma_t$ in which $R$ cannot be embedded,  and the torso $G_t$ of $V_t$ has an $\alpha$-near embedding $(\sigma_t,G_{t0},A_t,\V_t,\emptyset)$ in $\Sigma_t$ with the following properties:
\begin{enumerate}[\rm (i)]
	\item 
		The vortices $V\in\V_t$ have decompositions of width at most $\alpha$ satisfying {\rm (ii)} below.
	\item 
		For every $t'\in T$ with $tt'\in E(T)$ and $t\in rTt'$ the overlap $V_t\cap V_{t'}$ is contained in $A_{t'}$, 
		and $(V_t\cap V_{t'})\setminus A_t$ is contained either in a part $X_{tt'}$ of a vortex decomposition from {\rm (i)} 
		or in a subset $X_{tt'}$ of $V(G_{t0})$ that spans in $G_{t0}$ either a $K_1$ or a $K_2$ 
		or a $K_3$ bounding a face of $G_{t0}$ in $\Sigma_t$.
		In the latter case, $G_{t'}-A_t$ is properly attached to $G_{t0}$.
	\item
		If $t=r$, then $Z\subseteq A_r$. We say that the part $V_r$ 
		(with the chosen near-embedding of $G_r$) \emph{accomodates} $Z$.
\end{enumerate}
\end{thm}

We remark that the statement about $Z$ in Theorem~\ref{thm:extended1.3} only serves a technical purpose, to facilitate induction. 
The main difference between Theorem~\ref{thm:extended1.3} and Theorem~\ref{thm:3.1} is that the vortex decompositions required in 
Theorem~\ref{thm:extended1.3} have bounded width, while those in Theorem~\ref{thm:3.1} are only required to have bounded adhesion. 
It is this difference that requires the extra work when we deduce Theorem~\ref{thm:extended1.3} from Theorem~\ref{thm:3.1}: 
Starting from the $\alpha$-near embedding of $G$ provided by Theorem~\ref{thm:3.1}, we have to split off small vortices of large width, 
and large parts to decompose those parts of $G$ inductively.

\begin{proof}[Proof of Theorem \ref{thm:extended1.3}]
Applying Theorem \ref{thm:3.1} with the given graph $R$ yields two constants~$\hat\alpha$ and~$\hat \theta$. 
We may assume that $m$ is large enough that $\theta:=(m+2)/3\geq \max(\hat \theta, 3\hat\alpha+3)$ and $\theta$ is integral
and let $\alpha:=4\theta-2$.

The proof proceeds by induction on $|G|$, for these (now fixed) $R$, $m$ and $\alpha$. We may assume that $|Z|=m(=3\theta-2)$, since if it is smaller we can add arbitrary vertices  to $Z$.  (We may assume that such vertices exist, as the theorem is trivial for $|G|<\alpha$.)

We may assume that
\begin{equation}\label{claim:nosep}
\begin{minipage}[c]{0.8\textwidth}\em
There is no separation $(A,B)$ of $G$ of order at most $\theta$ such that both $|Z\cap A|$ and $|Z\cap B|$ are of size at least $|A\cap B|$.
\end{minipage}\ignorespacesafterend 
\end{equation}
Otherwise, let $Z_A:=(Z\setminus B)\cup (A\cap B)$. By assumption, $|A\cap B| \leq |Z\cap B|$, so $|Z_A| = |Z\setminus B| + |A \cap B| \leq |Z|=m$. 
We apply our theorem inductively to $G[A]$ and $Z_A$, which yields a tree-decomposition of $G[A]$ 
such that the torso of its root part $G_A$ has its apex set in a suitable near-embedding contain $Z_A$. 
Similarly, we apply the theorem to $G[B]$ and $Z_B:=(Z\setminus A)\cup (A\cap B)$.  
We combine these two tree-decompositions by joining a new part $Z\cup(A\cap B)$ to both $G_A$ and $G_B$ 
and obtain a tree-decomposition of $G$ with the desired properties of the theorem:  
The new part, which we make the root, contains at most $|Z|+|A\cap B|\leq 4\theta-2$ vertices, so 
all these can be put in the apex set of an $\alpha$-near embedding. 
Finally, the new part contains $Z$, and the new decomposition inherits all the remaining desired 
properties from the decomposition of $G[A]$ and $G[B]$. This proves (\ref{claim:nosep}).

Let $\T$\?{$\T$} be the set of separations $(A,B)$ of $G$ of order less than $\theta$ such that $|Z\cap B| > |Z\cap A|$.
Let us show that
\begin{equation}\label{claim:tistangle}
\begin{minipage}[c]{0.8\textwidth}\em
$\T$ is a tangle of $G$ of order $\theta$.
\end{minipage}\ignorespacesafterend 
\end{equation}
By (\ref{claim:nosep}) and our assumption that $|Z|=m=3\theta-2$, for every separation $(A,B)$ of $G$ of order less than $\theta$ 
exactly one of the sets $Z\cap B$ and $Z\cap A$ has size less than $\theta$.
This implies both conditions from the definition of a tangle.

From (\ref{claim:nosep}) and the definition of $\T$ we conclude
\begin{equation}\label{claim:smallcontainslittle}
\begin{minipage}[c]{0.8\textwidth}\em
$|Z\cap A|< |A\cap B|$ for every $(A,B)\in\T$.
\end{minipage}\ignorespacesafterend 
\end{equation}
As $\theta \geq \hat \theta$, Theorem~\ref{thm:3.1}  gives us an $\hat \alpha$-near embedding 
$(\sigma,G_0,\hat A,\hat \V,\hat \W)$\?{$\sigma,G_0,\hat A,\hat \V,\hat \W$} of $G$ in some surface $\Sigma$ that captures $\T$. 
Our plan now is to split $G$ at separators consisting of apex vertices, of society vertices of vortices in $\hat \W$, or 
single parts of vortex decompositions of vortices in~$\hat\V$. We shall retain intact a part of $G$ that contains $G_0$, 
and which we know how to embed $\alpha$-nearly; this part is going to be a part of a new tree-decomposition. 
For the subgraphs of $G$ that we split off we shall find tree-decompositions inductively, and eventually we shall combine all these 
tree-decompositions to one tree-decomposition of $G$ that satisfies our theorem. 

By Lemma~\ref{lem:properlyattached}, we may assume that large vortices contain no edges between consecutive society vertices, 
and that all small vortices are properly attached to~$G_0$. Let us consider such a vortex $(G_i,\Omega_i)\in \hat \W$.  
Since our embedding captures $\T$, the separation $(V(G_i)\cup \hat A, V(G\setminus (G_i\setminus \Omega_i))\cup \hat A)$, whose order 
is at most $3+|\hat A|< \theta$, lies in $\T$. By (\ref{claim:smallcontainslittle}), $G_i$ contains at most $2+|\hat A|$ vertices of $Z$. 
Thus, $Z_i':=\Omega_i\cup \hat A \cup (Z\cap G_i)$ has size at most $5+2\hat \alpha \leq m$.  
We apply our theorem inductively to the smaller graph $G[V(G_i)\cup \hat A]$ with $Z_i'$. 
Let \?{$H^i$}$H^i$ be the torso of the root part of the resulting tree-decomposition $(T^i,\Hc^i)$, \?{$T^i$, $\Hc^i$} the one that accomodates~$Z_i'$. 
Recall that $G_i$ was properly attached to $G_0$. The $\Omega_i$-paths witnessing this have no vertices in $\hat A$, and by replacing any $H^i$-subpaths 
they contain with torso-edges of $H^i$, we can turn them into paths witnessing that also $H^i-\hat A$ is properly attached to $G_0$. 

For every vortex $(G_i,\Omega_i)\in\hat\V$, with $\Omega_i=\{w_1^i,\ldots,w_{n(i)}^i\}$ say, let us choose a fixed decomposition $(\hat X_1^i,\ldots,\hat X_{n(i)}^i)$ of adhesion at most $\hat\alpha$. We define
\[
	X_j^i:=\begin{cases}
		\big(\hat X_1^i \cap \hat X_2^i\big) & \text{for } j = 1\\
		\big(\hat X_j^i \cap (\hat X_{j-1}^i \cup \hat X_{j+1}^i) \big) & \text{for }1< j < n(i)\\
		\big(\hat X_{n(i)}^i \cap \hat X_{n(i)-1}^i \big) & \text{for } j = n(i)
	\end{cases} 
\]
By \?{$G_i^-$}$G_i^-$ we denote the graph on $X_1^i\cup\ldots\cup X_{n(i)}^i$ in which every $X_j^i$ induces a complete graph but no 
further edges are present. 
Now, as the adhesion of $(G_i,\Omega_i)$ is at most $\hat \alpha$, every $X_j^i$ contains at most $2\hat \alpha$ vertices and thus, 
$(X_1^i,\ldots,X_{n(i)}^i)$ is a decomposition of the vortex\?{$V_i^-$} $V_i^- := (G_i^-,\Omega_i)$ of depth at most $2\hat \alpha\leq \alpha$. 
Let~$\V$\?{$\V$} denote the set of these new vortices.

For every $j=1,\ldots,n(i)$, the pair 
\[\big(\hat X_j^i\cup \hat A, (V(G) \setminus (\hat X_j^i\setminus X_j^i))\cup \hat A\big)\] 
is a separation of order at most $|X_j^i\cup \hat A|\leq 2\hat \alpha + \hat \alpha \leq \theta$. As before, our embedding captures $\T$, 
so the separation lies in $\T$. By (\ref{claim:smallcontainslittle}), at most $\theta-1$ vertices from $Z$ lie in $\hat X_j^i \cup \hat A$. 
Let $Z_{ij}':= X_j^i\cup \hat A \cup (Z\cap \hat X_j^i)$. This set contains at most $2\theta-1\leq m$ vertices and, as before, we can apply our 
theorem inductively to the smaller graph $G[\hat X_j^i \cup \hat A]$ with $Z_{ij}'$. 
We obtain a tree-decomposition $(T_j^i,\Hc_j^i)$\?{$T_j^i$} of this graph, with the root torso $H_j^i$\?{$H_j^i$} accomodating~$Z_{ij}'$.  

Now, with $V_0:=V(G_0)\cup \hat A$, we can write 
\[
G=G[V_0]\cup\big(\bigcup \W\big) \cup \big( \bigcup \{ G[\hat X_j^i] : V_i\in\V, 1\leq j \leq n(i)\}\big).
\]
Let us now combine our tree-decompositions of the vortices in $\W$ and the graphs $G[\hat X_j^i]$ to a tree-decomposition of $G$: 
We just add a new tree vertex $v_0$ representing $V_0$ to the union of all the trees $T^i$ and $T_j^i$, and add edges from $v_0$
to every vertex representing an $H^i$ or an $H_j^i$ we found in our proof.

We still have to check that the torso of the new part $V_0$ can be $\alpha$-nearly embedded as desired. But this is easy: 
Let $G'_0$\?{$G'_0$} be the graph obtained from $G_0$ by adding an edge $xy$ for every two nonadjacent vertices $x$ and $y$ 
that lie in a common vortex $V\in \W$. We can extend the embedding $\sigma:G_0\hookrightarrow \Sigma$ to an embedding 
$\sigma':G'_0\hookrightarrow\Sigma$ by mapping the new edges disjointly to the discs $D(V)$.
Then $G':=G'_0\cup \bigcup G_i^-$ is the torso of $V_0$ in our new tree-decomposition, and $(\sigma',G_0',\hat A\cup Z,\V,\emptyset)$ 
is an $\alpha$-near embedding of $G'$ in $\Sigma$ whose apex set contains~$Z$. \end{proof}

As noted, Theorem~(1.3) of~\cite{GM16} is a direct result from Theorem~\ref{thm:extended1.3}:

\begin{cor}\label{thm:1.3}
For every nonplanar graph $R$ there exists an integer $\alpha$ such that every graph with no $R$-minor 
has a tree-decomposition $(V_t)_{t\in T}$ such that for every $t\in T$ there is a surface $\Sigma_t$ 
in which $R$ cannot be embedded but in which the torso $G_t$ corresponding to $t$ has an $\alpha$-near embedding $(\sigma_t, G_{t0},A_t,\V_t,\emptyset)$.
\end{cor}


\section{Graphs on Surfaces}
\label{sec:tools}


In this section, we collect results about graphs embedded in surfaces. Except for the last one, these results are not directly related to near-embeddings.
Our first tool is the {\em grid-theorem} from~\cite{GM5}; see~\cite{diestel} for a short proof.

\begin{thm}\label{thm:gridthm}
For every integer $k$ there exists an integer $f(k)$ such that every graph of tree-width at least $f(k)$ contains a wall of size at least $k$. 
\end{thm}

Every large enough wall embedded in a surface contains a large flat subwall:

\begin{lem}\label{lem:flatwall}
For all integers $k,g$ there is an integer $\ell =\ell(k,g)$ such that any wall of size $\ell$ embedded in a surface of genus at most~$g$ contains a flat wall of size~$k$.
\end{lem}

\begin{proof}
Let $\ell$ be chosen large enough that every $\ell$-wall contains $g+1$ disjoint $k+1$-walls. By \cite[Lemma~B.6]{diestel}, any $\ell$-wall $H$ in a surface $\Sigma$ of genus~$g$ contains a $k$-wall $H'$ each of whose bricks bounds an open disc in~$\Sigma$. If none of these open discs contains a point of $H$, the wall $H'$ is flat. Otherwise, the disc containing a point of $H$ contains all the other $k+1$-walls we considered, and thus, all these are flat.
\end{proof}

\begin{lem}\label{lem:torsotwlarge}
Let $G$ be a graph of tree-width at least~$w$. Then in every tree-decomposition of $G$ the torso of at least one part also has tree-width at least~$w$.
\end{lem}

\begin{proof}
If every torso has a tree-decomposition of width at most $w-1$, we can use \cite[Lemma~12.3.5]{diestel} to combine these into a tree-decomposition of $G$ of width at most~$w-1$.
\end{proof}

The following lemma is a direct corollary of Lemmas~B.4 and~B.5 from \cite{diestel}.

\begin{lem}\label{lem:genredcurves}
For every surface $\Sigma$, every closed curve $C \subset \Sigma$ that does not bound a disc in $\Sigma$ is genus-reducing.
\end{lem}

Let $\Sigma$ be a (closed) surface and $G$ be a graph embedded in $\Sigma$. For a face $f$ of $G$, let $S$ be the set of vertices that lie on $\boundary f$. If we delete $S$ and add a new vertex $v$ to $G$ with neighbours $N(S)$, we obtain a graph $G'$. It is easy to see that we can extend the induced embedding of $G-S$ to an embedding of $G'$. We say that the graph $G'$ embedded in $\Sigma$ was obtained from~$G$ by \emph{contracting $f$ to $v$}.

The following lemma is from Demaine and Hajiaghayi \cite{DemaineHajiaghayi08}.
\begin{lem}\label{lem:contractingface} For every two integers $t$ and $g$  there exists an integer $s = s(t,g)$ such that the following holds. Let $G$ be a graph of tree-width at least $s$ embedded in some surface $\Sigma$ of genus $g$. If $G'$ is obtained from $G$ by contracting a face to a vertex, then $G'$ has tree-width at least $t$.
\end{lem}

%

Our next lemma is due to Mohar and Thomassen \cite{graphsonsurfaces}:

\begin{lem}\label{lem:conccycles}
Let $G\hookrightarrow\Sigma\neq S^2$ be an embedding of representativity at least $2k+2$ for some $k\in\N$. Then, for every face $f$ of $G$ in $\Sigma$ there are $k$ concentric cycles $(C_1,\ldots,C_k)$ in $G$ such that $f\subseteq \interior D(C_k)$.
\end{lem}

For an oriented curve $C$ and points $x,y\in C$ we denote by $xCy$\?{$xCy$} the subcurve of $C$ with endpoints $x,y$ that is oriented from $x$ to $y$. 
For a graph $G$ embedded in a surface $\Sigma$, a face $f$ of $G$, and a closed curve $C$ in $\Sigma$, let $\C(C,f)$\?{$\C(C,f)$} denote the number of components of $C\cap f$.

\begin{lem}\label{lem:curvehitsvortexonce}
Let $G$ be a graph embedded in a surface $\Sigma$, and let $F$ be the set of faces. For an integer $r>0$, consider all $G$-normal, genus-reducing curves $C$ in $\Sigma$ that satisfy $|C\cap G|<r$. Let $C$ be chosen so that $\sum_{f\in F}\C(C,f)$ is minimal. Then, $\C(C,f)\leq 1$ for all $f\in F$.
\end{lem}
\begin{proof}
Suppose there is an $f\in F$ with $\C(C,f)>1$. Then there is a component $D$ of $f-C$ whose boundary $\boundary D$ contains two distinct components of $f\cap C$,%
   \COMMENT{}
say the interiors of disjoint arcs $xCy$ and $zCw$ following some fixed orientation of $C$. Then $x,y,z,w$ appear in this (cyclic) order on $C$. 

Join $x$ to $z$ by an arc $A$ through $D$. Then $C_w:=xAzCx$ and $C_y:=zAxCz$ are closed curves in $\Sigma$ meeting precisely in $A$. 
Each of them meets $G$ in fewer vertices than $C$ does, so neither $C_w$ nor $C_y$ are genus reducing. 
This implies by Lemma~\ref{lem:genredcurves} that $C_w$ and $C_y$ bound discs $D_w$ and $D_y$ in $\Sigma$. 
If $D_w$ contains a point of $C_y$ then $C_y\setminus A\subseteq D_w$ and hence $C\subseteq D_w$. 
But then $C$ bounds a disc contained in $D_w\subseteq \Sigma$ (by the Jordan curve theorem), which contradicts our assumption that $C$ is genus-reducing in $\Sigma$.
Hence $D_w\cap C_y=\emptyset$, and similarly $D_y\cap C_w=\emptyset$. This implies that $\closure{D_{w}} \cap \closure{D_{y}} = A$. 
But then $\closure{D_w} \cup \closure{D_y}$ is a closed disc in $\Sigma$ bounded by $C$, a contradiction as earlier.
\end{proof}

Whenever there are cycles enclosing a vortex $V$, we can find cycles tightly enclosing $V$:
\begin{lem}\label{lem:tightcycles}
For an integer $\alpha>0$, let $(\sigma,G_0,A,\V,\W)$ be  an $\alpha$-near embedding of some graph $G$ in a surface $\Sigma$ and let $C_1,\ldots,C_n$ be cycles enclosing a vortex $V\in \V$. Then, there are $n$ cycles $C'_1,\ldots,C'_n$ in $G_0$ that enclose $V$ tightly, 
such that $D(C'_1)\subseteq D(C_1)$.
\end{lem}
\begin{proof}
Let us write $D_k:=D(C_k)$ for $1\leq k \leq n$ and $D_{n+1}:=D(V)$ and $V(C_{n+1}):=\Omega(V)$. 
Suppose there is a cycle $C\subseteq G'_0\cap \closure{D_k} \setminus \closure{D_{k+1}}$ for some $1\leq k \leq n$ such that $C\neq C_k$. 
Then, we can replace $C_k$ by $C$ and obtain a new set of cycles in $G_0$ enclosing $V$. 
By this replacement, we reduce the number of vertices and edges of $G'_0$ in $\closure{D_k}$, 
so we can repeate this step only finitely often. 
We may assume that $C_1,\ldots,C_n$ were chosen so that such that a replacement is not possible. 

We claim that these cycles enclose $V$ tightly.  To see this, consider
for a vertex $v\in V(C_k)$ the set $F$ of all faces $f$ of $G_0'$ with $f\subseteq D(C_k)$ and $v\in \boundary f$. 
For every two neighbours $x,y$ of $v$ that lie on the boundary of the same face $f\in F$, 
there is a path in $G'_0\cap\boundary f$ linking $x,y$ and avoiding $v$. 
Therefore, there is a face $f_v\in F$ such that $\boundary f_v$ contains a vertex $v'\in V(C_{k+1})$\footnote{Indeed, 
thickening $G'_0$ in $\Sigma$ turns $\closure f$ into a compact surface with boundary. By the classification of these surfaces, the component
of $\boundary f$ meeting the thickened vertex $v$ is a circle, which defines a closed walk in $G'_0$. This walk contains the desired path.
}:
otherwise, $\bigcup\{\boundary f : f\in F\}$ would contain a path between the two neighbours $v^-$, $v^+$ of $v$ in $C_k$, 
that avoids $v$. 
Substituting this path for the path $v^-vv^+$ in $C_k$ then turns $C_k$ into a cycle in $G'_0\cap (\closure D_k \setminus \closure D_{k+1})$ avoiding $v$, 
contradicting the choice of the $C_i$. 
Similarly, every edge $e$ of $C_k$ lies on the boundary of a face $f$ of $G'_0$ that also contains a vertex $v'$ of $C_{k+1}$. 
Then every inner point $x$ of $e$ can be linked to $v'$ by a curve through $f$.
By induction on $n-k$, we may assume that, unless $k=n$ and $v'\in \Omega(V)$, there is a curve $C$ linking $v'\in  V(C_{k+1})$ 
to some $w\in \Omega(V)$, with $|C\cap G'_0|\leq n-(k+1)+2$. We extend this curve by a curve in $f$ from $v$ or $x$ to $v'$ which gives us a curve as desired.
\end{proof}


\section{Taming a Vortex}
\label{sec:modifying}


In this section we describe how to obtain a new (and simpler) near-embedding from an old one if we are allowed to move a bounded number of vertices from the embedded part of the graph to the apex set. 
For example, we might reduce the number of large vortices by combining two of them, or reduce the genus of the surface by cutting along a genus-reducing curve. 

\begin{lem}\label{lem:mergevortices}
Let $(\sigma,G_0,A,\V,\W)$ be an $(\alpha_0,\alpha_1,\alpha_2)$-near embedding of a graph $G$ in a surface $\Sigma$.
If there are two vortices $V,W\in \V$ of length at least 4 
and a $G$-normal curve $C$ in $\Sigma$ from $D(V)$ to $D(W)$ that meets $G$ in at most $d$ vertices, 
then there is a vertex set $A'\subseteq V(G_0)$ of size $|A'|\leq 2\alpha_2+d$ 
and a vortex $V'\subseteq G-A'$ such that $G$ has an $(\alpha_0+2\alpha_2+2+d,\alpha_1-1,\alpha_2)$-near embedding 
\[
	(\sigma|_{G_0-A'},G_0,A\cup A',\V', \W')
\]
in $\Sigma$ with $\V' \subseteq (\V \setminus \{V,W\})-A'\cup \{V'\}$ and 
$\W' = \W-A'\cup \{V-A': V\in\V, |\Omega(V)\setminus A'| \leq 3\}$.
\end{lem}
\begin{proof}
Let us choose decompositions $(X_1,\ldots,X_n)$ of $V$ and $(Y_1,\ldots,Y_m)$ of $W$ of adhesion at most $\alpha_2$, where 
$\Omega(V)=(v_1,\ldots,v_n)$ and $\Omega(W)=(w_1,\ldots,w_m)$. 
By slightly adjusting $C$ we may assume that the endpoints of $C$ are vertices $v_k$ and $w_\ell$, so that $C \cap D(V)=\{v_k\}$ 
and $C\cap D(W)=\{w_{\ell}\}$ for some indices $1\leq k\leq n$ and $1\leq \ell \leq m$. 
Let $S$ be the set of vertices in $\Sigma$ on $C$. 
By fattening $C$ to a disc $D$ we obtain a closed disc $D':=\closure{D\cup D(V)\cup D(W)}$ such that 
$D'\cap (G_0-S)= (\Omega(V)\cup\Omega(W))\setminus\{v_k,w_\ell\}$. By reindexing if neccesary we may assume that the orientations 
of $\partial D'$ induced by $\Omega(V)\setminus \{v_k\}$ and $\Omega(W)\setminus\{w_\ell\}$ agree.

Let $X:=(X_k\cap X_{k+1})$ if $k<n$ and $X:=\{v_k\}$ if $k=n$ and let $Y:=(Y_\ell\cap Y_{\ell+1})$ if $\ell<m$ 
and $Y:=\{w_\ell\}$ if $\ell=m$. Note that $|X|\leq \alpha_2$ and $|Y|\leq \alpha_2$. 
If $k\neq 1$, let $X'_{k-1}:=(X_{k-1}\cup X_k)\setminus X$ and $X'_i:=X_i\setminus X$ for all $i \not \in \{k-1,k\}$. 
If $k=1$, let $X'_2:=(X_1\cup X_2)\setminus X$ and $X'_i:=X_i\setminus X$ for all $i\geq 3$. 
Define sets $Y'_j$ analoguosly with $\ell$ and $m$ replacing $k$ and $n$. 
Finally, let $A':= S\cup X\cup Y$ and $G':=(V\cup W)-A'$. Then for
\[	
	\Omega':=(v_{k+1},\ldots,v_n,v_1,\ldots,v_{k-1},w_{\ell+1},\ldots,w_m,w_1,\ldots,w_{\ell-1})
\]
the tuple $V':=(G',\Omega')$ is a vortex with a decomposition
\[
 	(X'_{k+1},\ldots,X'_n,X'_1,\ldots,X'_{k-1}, Y'_{\ell+1},\ldots,Y'_m,Y'_1,\ldots,Y'_{\ell-1})
\] 
of adhesion at most $\alpha_2$. Now it is easy to see that $A'$ satisfies the conditions as desired.
\end{proof}

Our next lemma shows how to make vortices linked. The techniques used in its proof originate from \cite{DGJT} and were extended by Geelen and Huynh \cite{GH}.
\begin{lem}\label{lem:vorticeslinked}
Let $(\sigma,G_0,A,\V,\W)$ be  an $(\alpha_0,\alpha_1,\alpha_2)$-near embedding of a graph $G$ in a surface $\Sigma$ 
such that every small vortex $W\in\W$ is properly attached.
Moreover, assume that
\begin{enumerate}[\rm (i)]
\item 
For every vortex $V\in \V$ there are $\alpha_2+1$ concentric cycles $C_0(V),\ldots,C_{\alpha_2}(V)$ in $G'_0$ tightly enclosing $V$.
\item For distinct vortices $V,W\in\V$, the discs $\closure{D(C_0(V))}$ and $\closure{D(C_0(W))}$ are disjoint.
\end{enumerate}
Then there is a graph $\tilde G_0\subseteq G_0$ containing $G_0\setminus \Big (\bigcup_{V\in\V}D(C_0(V))\Big)$, 
a set $\tilde A\subseteq V(G)\setminus V(\tilde G_0)$ of size $|\tilde A|\leq\tilde\alpha:= \alpha_0+ \alpha_1(2\alpha_2+2)$, 
and sets $\tilde\V$ and $\tilde\W\subseteq \W$ of vortices
such that, with $\tilde \sigma:= \sigma|_{\tilde G'_0}$, the tuple $(\tilde\sigma,\tilde G_0,A\cup\tilde A,\tilde\V,\tilde\W)$ 
is an  $(\tilde\alpha,\alpha_1,\alpha_2+1)$-near embedding of $G$ in $\Sigma$ such that every vortex 
$\tilde V\in \tilde\V$ satisfies condition {\rm (\ref{prop:linked})} of the definition of $(\beta, r)$-rich,
and $D(\tilde V)\supseteq D(V)$ for some $V\in\V$.
\end{lem}
\begin{proof}
We will convert the vortices in $\V$ into linked vortices one by one, so let us focus on one vortex $V\in\V$. The idea is as follows: we delete one vertex from each of the enclosing cycles, which gives us a set of $\alpha_2+1$ disjoint paths. If necessary, we also delete an adhesion set of $V$ which allows us to assume that the paths are `aligned' to the vortex. Then, we `push' these paths as far into the vortex as possible. As the adhesion of the vortex is bounded by $\alpha_2$, at least one of the paths remains entirely in the surface. The vertices of the innermost such path, later denoted by $P_0$, become the society vertices of our new vortex, and the shifted path system shows that this new vortex is linked.

By assumption, $V$ has a decomposition $(X'_1,\ldots,X'_{n'})$ with adhesion sets $Z'_i:=X'_i\cap X'_{i+1}$ of size at most $\alpha_2$, for all $i<n'$. Pick a vertex $v\in C_0(V)$. As $C_0(V),\ldots,C_{\alpha_2}(V)$ enclose $V$ tightly, there is a curve $C$ from $v$ to $\Omega(V):=\{w'_1,\ldots,w'_{n'}\}$ that contains at most $\alpha_2+2$ vertices of $G'_0$. Let $S$\?{$S$} denote the set of these vertices. Clearly, $S$ consists of exactly one vertex from each $C_i(V)$, $0\leq i \leq \alpha_2$ and one society vertex $w'_j$ of $V$.

Put $n:=n'-1$ and $Z'_{n'}:=\emptyset$ and let $Z:=Z'_j\cup\{w'_j\}$. If $j=1$ let 
\[
	(X_1,\ldots,X_n):=\left( (X'_1\cup X'_2)\setminus Z, X'_3\setminus Z, \ldots, X'_{n'}\setminus Z\right).
\]
If $j>1$, let
\[
	(X_1,\ldots,X_n):=(X'_{j+1}\backslash Z,X'_{j+2}\backslash Z,\ldots,X'_{n'}\backslash Z,X'_1\backslash Z,\ldots,(X'_{j-1}\cup X'_j)\backslash Z).
\]
Then $(X_1,\ldots,X_n)$ is a decomposition of adhesion at most $\alpha_2$ of the vortex $V-Z$ taken with respect to the society $(w_1,\ldots,w_n)$ defined by $w_i:=X_i\cap \Omega(V)$ for all $i$. For $i<n$, let $Z_i:=X_i\cap X_{i+1}$. 

Recall that the linear ordering of $\Omega$ is induced by an orientation of the disc $\closure{D(V)}$. 
The extension of this orientation to $\closure D(C_i)$ induces a cyclic ordering on $V(C_i)$, for each $0\leq i \leq \alpha_2$, 
in which we let $x_i$ denote the successor, and $y_i$ the predecessor of the unique vertex in $S\cap V(C_i)$. 
Let $X:=\{x_0,\ldots, x_{\alpha_2}\}$\?{$X,Y$} and $Y:=\{y_0,\ldots,y_{\alpha_2}\}$.
Now we delete $S\cup Z$, a set of at most $2\alpha_2+1$ vertices, and put
\[
	 G':=\Big(\big(G'_0\cap \closure{D(C_0(V))} \big)\cup V\Big) - (S\cup Z).
\]
Clearly, the graph $G'$ still contains a set of $\alpha_2+1$ disjoint $X$--$Y$ paths. Let us show  that 
\begin{equation}\label{claim:p0insurface}
\begin{minipage}[c]{0.8\textwidth}\em
For every set $\P$ of $\alpha_2+1$ disjoint $X$--$Y$-paths in $G'$, the path $P_0$ starting in $x_0$ lies in~$G'_0$.
\end{minipage}\ignorespacesafterend 
\end{equation}
Otherwise, let $w_q$ be the vertex of $P_0$ preceding its first vertex in $V-\Omega(V)$. 
As the subpath $P_0w_q$ of $P_0$ lives entirely in the plane graph $G_0'\cap \closure{D(C_0(V))}$, 
the set $V(P_0w_q)\cup Z_q$	separates $X$ from $Y$. Thus, all $\alpha_2+1$ paths in $\P$ have to pass through $Z_q$, 
a set of at most $\alpha_2$ vertices, a contradiction. This proves (\ref{claim:p0insurface}).

By planarity, (\ref{claim:p0insurface}) implies that the paths in $\P\setminus\{P_0\}$ cannot cross $P_0$, so $P_0$ ends in $y_0$. 
Together with $v$ and the edges $x_0v$ and $vy_0$ the path $P_0$ forms a cycle in $G'_0$;
for our original set $\P$, this is the cycle $C_0$, but for every $\P$ satisfying (\ref{claim:p0insurface}) its $x_0$--$y_0$ path $P_0$ defines such a cycle in $G'_0$.
This cycle bounds a disc $D(\P)$ in $\Sigma$ containing $\Omega(V)$, and we define
\[
	G(\P):=\Big(\big(G'_0\cap D(\P)\big)\cup V\Big) - (S\cup Z).
\]
Clearly, $G(\P)$ contains the paths from $\P$.

Let us choose $\P$ as in (\ref{claim:p0insurface}) with $G(\P)$ minimal, and let the vertices of $P_0$ be labeled $p_0, \dots, p_r$.
Then we have the following:
\begin{equation}\label{claim:p0separations}
\begin{minipage}[c]{0.8\textwidth}\em
For every vertex $p_i\in V(P_0)$ there is a set $T$ of $\alpha_2+1$ vertices of $G(\P)$ that contains $p_i$ and separates $X$ from $Y$ in $G'$.
\end{minipage}\ignorespacesafterend 
\end{equation}
Indeed, if $i\in\{0,r\}$ then $T\in \{X,Y\}$ will do so assume that $0<i<r$. By the minimality of $G(\P)$,
there is no set of $\alpha_2+1$ disjoint $X$--$Y$ paths in $G'-p_i$.
Hence by Menger's theorem, an $X$--$Y$ separator $T'$ of size at most $\alpha_2$ exists in $G'-p_i$.
Since $G(\P)-p_i\subseteq G'-p_i$ contains the $\alpha_2$ paths of $\P\setminus\{P_0\}$, we have $T'\subseteq V(G(\P))$ and $|T'|=\alpha_2$. 
Now $T:=T'\cup \{p_i\}$ is as desired. This completes the proof of (\ref{claim:p0separations}).

Let us pick for each $i=0,\ldots,r$ a separation $(A_i,B_i)$ of $G(\P)$ as in (\ref{claim:p0separations}), with $p_i\in T_i:=A_i\cap B_i$ and $|B_i|$ 
minimal such that $X\subseteq A_i$ and $Y\subseteq B_i$.%
   \COMMENT{} 
   Clearly, each $T_i$ contains exactly one vertex from each path in $\P$. Let us show the following:
\begin{equation}\label{claim:Bidecreasing}
\begin{minipage}[c]{0.8\textwidth}\em
$B_i\supsetneq B_j$ for all $0\leq i < j \leq r$.
\end{minipage}\ignorespacesafterend 
\end{equation}

Note first that $B_i\ni p_i \in P_0p_i\subseteq A_j\setminus B_j$, so it suffices to show that $B_i\supseteq B_j$.
Suppose this fails. Then $|B_j|\cap B_i|<|B_j$, which will contradict our choice of $(A_j,B_j)$ if we can show that we could have chosen
the separation $(A_i\cup A_j, B_i\cap B_j)$ instead of $(A_j, B_j)$. 
Clearly, $p_j\in p_jP_0\subseteq B_i\cap B_j$, since $i<j$. 
Moreover, the new separator $T_B:=(A_i\cup A_j)\cap (B_i\cap B_j)$ contains at least one vertex from each path of $\P$,
so $|T_B|\geq |\P|$. 
Likewise, the $X$--$Y$ separator $T_A:=(A_i\cap A_j)\cap (B_i\cup B_j)$ meets every path in $\P$,
so $T_A\geq |\P|$. But $|T_A|+|T_B|=|T_i|+|T_j|=2|\P|$.
Hence both inequalities hold with equality; in particular, $|T_B|=|\P|=\alpha_2+1$ as desired.
This proves (\ref{claim:Bidecreasing}).

We set $\Omega:=V(P_0)$ and $X_0:=A_0$ and $X_i:=A_i\cap B_{i-1}$ for $1<i<r$ and $X_r:=B_{r-1}$. 
It is easy to check that $(X_0,\ldots,X_r)$ is a linked path decomposition of the vortex $(G(\P),\Omega)$. 
Finally, consider all $W\in\W$ such that $\Omega(W)$ intersects $G(\P)$. 
If $\Omega(W)\subseteq G(\P)$, we add $W$ to $G(\P)$ and delete it from $\W$. 
Otherwise, any vertices of $\Omega(W)$ in $G(\P)$, and any edges of $G'_0$ between them, are vertices and edges of $P_0$. 
We then delete any such edges from $G(\P)$ and dent $D((G(\P),\Omega))$ a little, so that its boundary no longer meets the interior of such edges.
Since these $W$ were properly attached, such edges can be replaced on $P_0$ by paths through $W$. 
Hence, property (\ref{prop:linked}) from $(\beta,r)$-rich follows for the new vortex $G(\P)$.  
\end{proof}

\begin{lem}\label{lem:rephigh}
Let $z>0$ be an integer, and $(\sigma,G_0,A,\V,\W)$ an $(\alpha_0,\alpha_1,\alpha_2)$-near embedding of a graph $G$ in a surface $\Sigma$ such that every two vortices in~$\V$ have distance at least $z$ in $\Sigma$. If the representativity of $G'_0$ in $\Sigma$ is less than $z$, 
then there is a vertex set $A' \subseteq V(G_0)$ with $|A'|<a:=2\alpha_2+2+z$, 
such that one of the following statements holds:
\begin{enumerate}[a)]
\item There exists a set $\V'$ of vortices in $G$, a surface $\Sigma'$ with $g(\Sigma')<g(\Sigma)$ and an $(\alpha_0+a,\alpha_1+1,\alpha_2)$-near embedding
\[
	(\sigma',G_0-A',A\cup A',\V',\W-A')
\]
of $G$ in $\Sigma'$. 
\item There exists a separation $(A_1,A_2)$ of $G$ with $A_1\cap A_2=A'$ such that for $i=1,2$ there are surfaces $\Sigma_i$ and $(\alpha_0+a,\alpha_1,\alpha_2)$-near embeddings $(\sigma^i, G^i_0, A^i, \V^i, \W^i)$ of $G[A_i]$ into $\Sigma_i$ such that $g(\Sigma_i)<g(\Sigma)$.
\end{enumerate}
\end{lem}

\begin{proof}
Let $C$ be a genus-reducing curve in $\Sigma$ that hits less than $z$ vertices of $G'_0$. 
Let us assume that $C$ meets the open disc $D(V)$ of a large vortex $V$, the case when it does not is even easier. 
Note that $C$ cannot meet another large vortex, since the distance in $\Sigma$ of two large vortices is at least~$z$. 
By Lemma~\ref{lem:curvehitsvortexonce} we may assume that $C$ meets the face $f$ of $G'_0$ containing $D(V)$ 
in at most one component (of $C\cap f$). We can therefore modify $C$ so that there are society vertices $x,y\in\Omega(V)$ 
with $\boundary D(V) \cap C=\{x,y\}$. 
If $C$ enters $D(V)$ from a disc $D(W)$ of a small vortex $W$ with $x\in\Omega(W)=\{w_{i-1},w_i,w_{i+1}\}\subseteq \Omega(V)$,
where $w_{i-1},w_i,w_{i+1}$ are enumerated as in $\Omega(V)$,
we choose $C$ so that $x=w_{i+1}$.%
   \COMMENT{}
   After this modification, $C$ hits at most $z+2$ vertices. Deleting the two appropriate adhesion sets splits $V$ into two vortices $V',V''$, 
and using facts from elementary point-set topology such as in \cite[Chapter 4.1]{diestel}%
   \COMMENT{}
   we can partition $D(V)$ into two discs $D(V'),D(V'')$ to accomodate them. 
Let $A'$ be the union of the deleted adhesion sets and the vertices hit by $C$. This set contains up to $a$ vertices.

Now, $A'$ is a separator of $G$. We delete $C$ from $\Sigma$ and cap the holes of the resulting components; cf. \cite[Appendix B]{diestel}. 
If $\Sigma\setminus C$ has one component, statement (a) follows, otherwise (b) is true. 
\end{proof}

\begin{lem}\label{lem:pieslices}
Let $(\sigma,G_0,A,\V,\W)$ be an $(\alpha_0,\alpha_1,\alpha_2)$-near embedding of a graph $G$ in a surface $\Sigma$. 
Let $V\in\V$ be a vortex, and $(C_1,\ldots,C_\ell)$ cycles tightly enclosing~$V$. 
Then there is a set $\mathcal X$ of disjoint open discs with $\bigcup_{\Delta\in \mathcal X} \closure \Delta=\closure{D(C_1)}$ such that the following holds: 
For every disc $\Delta \in \mathcal X$ there are sets $S \subseteq V(G'_0)$ of size $|S|\leq 2\ell+2$ and $S'\subseteq V$ of size $|S'|\leq 2\alpha_2-2$ such that
\begin{itemize}
\item $S=G'_0\cap \boundary \Delta\subseteq V(C_1)\cup \ldots \cup V(C_\ell)\cup \Omega(V)$
\item $|S\cap C_i| \leq 2$  for each $i=1,\ldots,\ell$
\item $|S\cap \Omega(V)|\leq 2$
\item there is a separation $(X_1,X_2)$ of $G$ with $X_1\cap X_2 = S\cup S'$ and $G_0\cap X_1 = G_0\cap \Delta$.
\end{itemize}
\end{lem}

\begin{proof}
Pick a point $p\in D(V)$. Let $(x_0,\ldots,x_n)$ denote the vertices of $C_1(V)$ ordered linearly in a way compatible with a cyclic orientation of $C_1(V)$. 
For $1\leq i \leq n$, denote the edges $x_{i-1}x_i$ by $e_i$, and put $e_{n+1}:=x_nx_0$ and $e_0:=\emptyset$.
We will inductively define curves linking $x_0,\ldots,x_n$ to $p$. 
First, let us choose for all $i=0,\ldots,n$ a curve $L'_i$ linking $x_i$ to $p$ so that $L'_i\cap G'_0$ consists of exactly one vertex from 
each of the cycles $C_1,\ldots,C_\ell$ and one society vertex $w_i\in\Omega(V)$ and so that $L'_iw_i\cap D(V)=\emptyset$ and $w_iL'_i\subseteq \closure{D(V)}$.
Then put $L_0:=L'_0$, and for $i=1,\ldots,n$ define $L_i$ inductively as follows.
Let $z$ be the first point of $L'_i$ on $L_0\cup L_{i-1}$. If $z\in L_0$, let $L_i:=L'_iz$; otherwise let $L_i:=L'_izL_{i-1}$. 
Note that, for $1\leq i\leq n$ and every point $z\in L'_i$, $|L'_iz \cap (C_1\cup\ldots\cup C_\ell)|=k$ if and only if 
$z\in \closure{D(C_k)}\setminus \closure{D(C_{k+1})}$ where $D(C_{\ell+1}):=\emptyset$.

Elementary topology implies (as in the proof of Lemma~\ref{lem:tightcycles}) that $D(C_1)\setminus L_0$ has a unique component homeomorphic to an open disc $\Delta'_0$.
Assume inductively that, for some $1\leq i \leq n$, we have defined an open disc $\Delta'_{i-1}\subseteq D(C_1)$ whose boundary is contained in $L_0\cup L_{i-1} \cup C_1$,
contains $L_{i-1}$, and meets $C_1$ in exactly $C_1\setminus (e_0\cup \ldots \cup  e_{i-1})$. Then $\Delta'_{i-1}$ contains the interior of $L_i$, which
joins two points of $\boundary \Delta'_{i-1}$ and thus divides $\Delta'_{i-1}$ into two open discs $\Delta_i$ and $\Delta'_i$. We let $\Delta'_i$ be the disc whose boundary
satisfies for $i$ the requirements analogous to those made earlier for $i-1$ on $\Delta'_{i-1}$, 
and let $\Delta_i:=\Delta'_{i-1}\setminus \boundary \closure{\Delta'_i}$ be the other disc. 
Finally, put $\Delta_{n+1}:=\Delta'_n$, set $\mathcal X:=\{\Delta_1,\ldots, \Delta_{n+1}\}$, and let $\boundary \Delta_i\cap G_0=:S_i$ for all $i$.

Induction on $i$ shows that $S_i$ has exactly one edge on $C_1$ and otherwise lies in $L_0\cup L_i$, so $|S_i|\leq 2\ell+2$.
If $\Delta_i\cap D(V)=\emptyset$, then $S_i$ and $S'_i:=\emptyset$ are as desired. 
Otherwise, $\boundary\Delta_i$ meets $\closure{D(V)}$ in an arc linking distinct vertices $w_i, w_j$. 
Let $S'_i$ be the union of the corresponding adhesion sets $Z_i, Z_j$ in a vortex decomposition of $V$ of adhesion $\leq \alpha_2$.
Again, $S_i$ and $S'_i$ are as desired.
\end{proof}


\section{Streamlining Path Systems}
\label{sec:pathsystems}


In this section we provide tools that allow us to find path systems in near-embeddings that satisfy conditions (\ref{prop:linkagetogrid}) and (\ref{prop:orthogonalpaths})
in the definition of $(\beta, r)$-rich.

\begin{lem}\label{lem:concatpaths}
Let $G$ be a graph and $A,B,C$ subsets of $V(G)$ with $|B|=2k-1$ for some integer $k$. If $G$ contains a set $\P$ of $2k-1$ disjoint $A$--$B$ paths, 
and a set $\Q$ of $2k-1$ disjoint $B$--$C$ paths, then there are $k$ disjoint $A$--$C$ paths in $G$.
\end{lem}

\begin{proof}
Let $\P$ be a set of $2k-1$ disjoint $A$--$B$ paths, and let $\Q$ a set of $2k-1$ disjoint $B$--$C$ paths in $G$.
For every set $S\subseteq V(G)$ with $|S|<k$, at least $k$ paths in $\P$ and at least $k$ paths in $\Q$ avoid $S$. 
Two of these paths contain a common vertex of $B$, so $G-S$ contains a path from $A$ to $C$. The existence of $k$ disjoint $A$--$C$ paths now follows by Menger's theorem.
\end{proof}

\begin{lem}\label{lem:pathstoboundary}
Let $G$ be a graph embedded in a surface, let $H$ be a flat wall in $G$ of size $32k^2+r$ in $G$ for integers $k<r$, 
and let $\Omega$ be a subset of $V(G)$ avoiding $D(H)$ such that there are $16k^2$ disjoint paths from $\Omega$ to branch vertices of $H$. Then $H$ contains a wall $H_0$ of size $r$ and $k$ disjoint paths from $\Omega$ to branch vertices of $H_0$ that lie on the boundary cycle of~$H_0$. 
\end{lem}
\begin{proof}
Let $H_0$ be an $r$-wall in $H$ with $8k^2$ concentric cycles in $H$ enclosing $H_0$. Choose a set $\P$ of $16k^2$ disjoint paths from $\Omega$ to branch vertices of $H$ such that $|E(\P)\setminus E(H)|$ is minimal, and among these so that $\sum_{P\in\P}|P|$ is minimal.

We claim that no path $P\in\P$ meets $H_0$. Otherwise $P$ would meet each of our $8k^2$ concentric cycles $C\subseteq H$ without ending on $C$. By the choice of $\P$, this means that $P$ contains no branch vertex of $H$, but meets $C$ only inside one subdivided edge of $C$ before leaving it again.
By our first condition for the choice of $\P$, the branch vertices of $H$ that are the ends of this subdivided edge must each lie on another path from $\P$,
which must end there. So we have at least $16k^2$ paths from $\P$ other than $P$ ending at such branch vertices, which contradicts our assumption
that $|\P|=16k^2$.

As $|\P|=16k^2$, either at least $4k$ rows or at least $4k$ columns contain terminal vertices of paths in $\P$. 
In either case it is easy to see that at least half of these branch vertices (i.e. $\geq 2k$) can be linked disjointly to branch vertices on the boundary cycle of $H_0$.  Lemma~\ref{lem:concatpaths} completes the proof.
\end{proof}

Let $G$ be a graph and $X,Y\subseteq V(G)$ with $|X|=|Y|=:k$. An $X$--$Y$~\emph{linkage} in~$G$ is a set of $k$ disjoint paths in $G$ such that each of these paths has one end in $X$ and the other end in $Y$.

An $X$--$Y$~linkage $\P$ in  $G$ is \emph{singular}  if $V(\bigcup\P) = V(G)$ and $G$ does not contain any other $X$--$Y$ linkage.%
   \COMMENT{}
The next lemma will be used in the proof of Lemma \ref{lem:ortholink}. 

\begin{lem}\label{lem:singlink}
If a graph $G$ contains a singular $X$--$Y$ linkage~$\P$ for vertex sets $X,Y\subseteq V(G)$, then $G$ has path-width at most~$|\P|$.
\end{lem}

\begin{proof}
Let $\P$ be a singular $X$--$Y$ linkage in~$G$. Applying induction on~$|G|$, we show that $G$ has a path-decomposition $(X_0,\dots,X_n)$ of width at most~$|\P|$ such that $X\subseteq X_0$. For every $x\in V(G)$ let $P(x)$ denote the path $P\in\P$ that contains $x$. Suppose first that every $x\in X$ has a neighbour $y(x)$ in~$G$ that is not its neighbour on $P(x)$. Then $y(x)\notin P(x)$ by the uniqueness of~$\P$.%
   \COMMENT{}
   The digraph on $\P$ obtained by joining for every $x\in X$ the `vertex' $P(x)$ to the `vertex' $P(y(x))$ contains a directed cycle~$D$.%
   \COMMENT{}
   Let us replace in $\P$ for each $x\in X$ with $P(x)\in D$ the path $P(x)$ by the $X$--$Y$ path that starts in~$x$, jumps to $y(x)$, and then continues along~$P(y(x))$. Since every `vertex' of~$D$ has in- and outdegree~1 in $D$, this yields an $X$--$Y$ linkage with the same endpoints as $\P$ but different from~$\P$. This contradicts our assumption that $\P$ is singular. Thus, there exists an $x\in X$ without any neighbours in $G$ other than (possibly) its neighbour on~$P(x)$. Consider this~$x$.

If $P(x)$ is trivial, then $x$ is isolated in $G$ and $x\in X\cap Y$. By induction, $G-x$ has a path-decomposition $(X_1,\dots,X_n)$ of width at most $|\P|-1$ with $X\setminus\{x\}\subseteq X_1$.%
   \COMMENT{}
   Add $X_0:= X$ to obtain the desired path-decomposition of~$G$. If $P(x)$ is not trivial, let $x'$ be its second vertex, and replace $x$ in $X$ by $x'$ to obtain~$X'$. By induction, $G-x$ has a path-decomposition $(X_1,\dots,X_n)$ of width at most $|\P|$ with $X'\subseteq X_1$. Add $X_0:= X\cup\{x'\}$ to obtain the desired path-decomposition of~$G$.
\end{proof}
 
 Our next lemma is a weaker version of Theorem 10.1 of~\cite{JORG2}.
 
\begin{lem}\label{lem:ortholink}
Let $s$, and $t$ be positive integers with $s\ge  t $. Let $G'$ be a graph embedded in the plane, and let $X\subseteq V(G')$ be a set of $t$ vertices on a common face boundary of $G'$. Let $(C_1, \dots, C_s)$ be concentric cycles in $G'$, tightly enclosing $X$.
Let $G''$ be another graph, with $V(G')\cap V(G'') \subseteq V(C_1)$. Assume that $G'\cup G''$ contains an $X$--$Y$ linkage $\P$  with $Y \subseteq V(C_1)$.  Then there exists an $X$--$Y$ linkage $\mathcal{P}'$ in $G'\cup G''$ such that $\mathcal{P}'$ is orthogonal to $C_{t+1}, \dots, C_{s}$.  
\end{lem}

\begin{proof}
Assume the lemma is false, and let $G'$, $G''$, $\mathcal{P}$, and $(C_1, \dots, C_s)$ form a counterexample with the minimum number of edges. 
Then $G:=G'\cup G'' = \bigcup_{i=1}^s C_i \cup \mathcal{P}$, and $V(G)=V(\bigcup \P)$. (Delete isolated vertices if necessary).  
Let us show that, for all $P \in \mathcal{P}$ and for all $1 \le i \le s$, every component of $P \cap C_i$ is a single vertex. 
Indeed, if $P \cap C_i$ had a component  containing an edge $e$, then $G'/e$ would form a counterexample with fewer edges,
since any path using the contracted vertex $v_e$ would still form, or could be expanded to a path for our desired path system $\P'$.

Next, let us show that
\begin{equation}\label{claim:Gsingular}
\begin{minipage}[c]{0.8\textwidth}\em
$\P$ is singular. \end{minipage}\ignorespacesafterend 
\end{equation}
If there exists an $X$--$Y$~linkage $\overline{\mathcal{P}}$ distinct 
from $\mathcal{P}$, then at least one of the edges of $\mathcal{P}$ is not contained in 
$\overline{\mathcal{P}}$. 
Since, as noted above, the paths in $\P$ have no edges on $C_1,\ldots,C_s$,
the subgraph $\bigcup_{i=1}^s C_i \cup \bigcup \overline{\mathcal{P}}$  forms a counterexample
with fewer edges, a contradiction.  This proves~$(\ref{claim:Gsingular})$.

Our choice of the cycles $C_i$ as tightly enclosing $X$ implies at once:
\begin{equation}\label{claim:nobumps}
\begin{minipage}[c]{0.8\textwidth}\em
There is no subpath $Q$ of some path $P\in \mathcal{P}$ in $D(C_j)$ with both endpoints in $C_j$ for some $j$ and otherwise disjoint
from   $\bigcup_{i} V(C_i)$.
\end{minipage}\ignorespacesafterend 
\end{equation}
A \emph{local peak} of $\mathcal{P}$ is a subpath $Q$ of  a path $P\in \mathcal{P}$ 
such that $Q$ has both endpoints on $C_j$ for some $j>1$ and every internal vertex 
of $Q$ in $\left ( \bigcup_{i} V(C_i) \right)$ lies in $V(C_{j-1})$.  

Let us show the following:
\begin{equation}\label{cl:3}
\begin{minipage}[c]{0.8\textwidth}\em
$\P$ has no local peak
\end{minipage}\ignorespacesafterend 
\end{equation}
Suppose $Q=x\ldots y \subseteq P\in \P$ is a local peak, with endpoints in $C_j$ say, chosen 
so that $j$ is maximal. 

Let $xC_jy$ denote the subpath of $C_j$ such that the cycle $xC_jy \cup Q$ bounds a disc $D\subseteq D(C_{j-1})\setminus D(C_j)$.
If no interior vertex of $xC_jy$ lies on a path from $\P$, we can replace $Q$ by $xC_jy$ on $P$ and then contract this subpath of $P$, 
to obtain a counterexample with fewer edges.
Hence $xC_jy$ does have an interior vertex $z$ on a path $P'\in\P$.%
   \COMMENT{}
   Let $zP'z'$ be a minimal non-trivial subpath of $P'$ such that $z'\in \bigcup_iC_i$ (This exists, as $j>1$.) If $z'\in C_j$, then $zP'z'\subseteq D$ by  (\ref{claim:nobumps}).
We then repeat the argument, with the local peak $zP'z'$ instead of $Q$. This can happen only finitely often, 
and will eventually contradict the minimality of our counterexample.
We may thus assume that $z'$ cannot be chosen in $C_j$. Then $zP'z'\cap D = \emptyset$ and $z'\in C_{j+1}$,%
   \COMMENT{}
   and $zP'z'$ extends 
to a subpath $z''P'z'$ of $P'$ with $z''\in C_{j+1}$ and no vertex other than $z,z',z''$ in $\bigcup_i C_i$. 
This path is a local peak of $\P$ that contradicts our choice of $Q$ with $j$ minimal, completing the proof of $(\ref{cl:3})$.

An immediate consequence of 
$(\ref{claim:nobumps})$ and $(\ref{cl:3})$ is the following.  
For every $P \in \mathcal{P}$, let $x$ be the endpoint of $P$ in $X$ and 
let $y$ be the vertex of $V(C_1) \cap V(P)$ closest to $x$ on $P$.
Then the subpath $\overline{P}:=xPy$ of $P$
is orthogonal to the cycles $C_1, \dots, C_s$.  
In fact, $\overline{P} \cap C_i$ is a single vertex, for each $1 \le i \le s$.  

The final claim will complete our proof:
\begin{equation}\label{cl:4}
\begin{minipage}[c]{0.8\textwidth}\em
For every $P \in \mathcal{P}$, the path $P - \overline{P}$ does not meet $C_{t+1}$. 
\end{minipage}\ignorespacesafterend 
\end{equation}

To prove $(\ref{cl:4})$, suppose there exists $P \in \mathcal{P}$ such that 
 $(P - \overline{P}) \cap C_{t+1} \neq \emptyset$. 
As before, it follows now from $(\ref{claim:nobumps})$ and $(\ref{cl:3})$
 that $P - \overline{P}$ contains a subpath $Q$ from $C_{t+1 }$ 
to $C_1$ that is orthogonal to the cycles $C_{t+1}, C_{t}\dots, C_{1}$.  
Together with final segments of our paths $\overline P$ and the cycles $C_1,\ldots,C_{t+1}$, this path $Q'$ 
forms a subdivision of the $(t+1)\times (t+1)$ grid, which is well known to have path-width $t+1$.  This 
contradicts $(\ref{claim:Gsingular})$ and Lemma \ref{lem:singlink}, proving $(\ref{cl:4})$.  
\end{proof}


\section{Proof of the Main Result}
\label{sec:mainresult}

Before proceeding with the proof of Theorem \ref{thm:richthm}, we will need one more lemma.  A similar result can be found in \cite{DemaineHajiaghayi08}.
\begin{lem}\label{lem:G0haslargetw}
For every integer $t$ and all integers $\alpha,g>0$ there is an integer $s>0$ such that the following holds. Let $G$ be a graph of tree-width at least $s$ and $(\sigma,G_0,A,\V,\emptyset)$ an $\alpha$-near embedding of  $G$ in a surface $\Sigma$ of genus $g$ such that all vortices $V\in \V$ have depth at most $\alpha$. Then $G_0$ has tree-width at least $t$.
\end{lem}
\begin{proof}
Let $t$, $\alpha$, $g$ be given. By Lemma~\ref{lem:contractingface}, there is an integer $r$ such that for every graph $H$ of 
tree-width at least $r$ embedded in a surface $\Sigma$ of genus $g$ the contraction of $\alpha$ disjoint faces of $H$ to vertices 
leaves a graph of tree-width at least $t+\alpha$.  

Let $G$ be a graph as stated in the Lemma, of tree-width $s>\alpha r$. 
Let $G_0^+$\?{$G_0^+$} be the graph we obtain if for every  vortex $V\in\V$ with $\Omega(V)=(w_1,\ldots,w_n)$ say, 
we add to $G_0$ all edges $w_jw_{j+1}$ for $1\leq j \leq n$, where $n+1:=1$ if not already in $G_0$. 
Clearly, $\sigma$ can be extended to an embedding of $G_0^+$ by embedding the new edges in the corresponding discs $D(V)$. 
\begin{equation}\label{claim:twG_0+large}
\begin{minipage}[c]{0.8\textwidth}\em
The tree-width of $G_0^+$ is at least $r$.
\end{minipage}\ignorespacesafterend 
\end{equation}
Otherwise, choose a tree-decomposition $(V_{t})_{t\in T}$ of $G_0^+$ of width less than $r$. 
For every vortex $V\in \V$ choose a fixed decomposition $(X_1,\ldots,X_n)$ of depth at most $\alpha$. 
For every $t\in T$ define
\[
V'_t:=V_t\cup\bigcup_{V\in\V}\{X_j: w_j\in V_t\}
\]
Note that, as all vortices are disjoint and thus every vertex in $V_t$ can be a society vertex of at most one vortex, we have $|V'_t|\leq\alpha |V_t|<\alpha r$. 
We claim that $(V'_t)_{t\in T}$ is a tree-decomposition of $G\cup G_0^+$.  
To see this, pick a vertex  $v\in V_{t_1}\cap V_{t_3}$ for distinct $t_1,t_3\in T$.  
We have to show that $v\in V'_{t_2}$ for all $t_2\in t_1Tt_3$. 
Let us assume that $v\notin V(G_0^+)$ as the other case is easy. 
By construction, there is a vortex $V$ with $\Omega(V)=(w_1,\ldots,w_n)$ such that for some $w_j,w_k\in \Omega(V)$, 
we have $v\in X_j\cap X_k$ and $w_j\in V_{t_1}$ and $w_k\in V_{t_3}$. 
We may assume without loss of generality that $j<k$. 
By construction, $G_0^+$ contains path $w_jw_{j+1}\ldots w_k$. 
As $V_{t_2}$ separates $V_{t_1}$ from $V_{t_3}$ in $G\cup G_0^+$, 
there is a vertex $w_\ell\in V_{t_2}$ for some $j\leq\ell\leq k$. 
Then $v\in X_\ell$, since $(X_1,\ldots,X_n)$ is a path-decomposition, so $v\in V_{t_2}$ as desired.

Clearly, $(V_t)_{t\in T}$ is a tree-decomposition of $G$ as well, 
but it has width at most $\alpha r$, a contradiction to our choice of $G$. This proves (\ref{claim:twG_0+large}). 

For every vortex $V\in \V$ there is a face $f\subseteq D(V)$ of $G_0^+$ with $\Omega(V)= \boundary f\cap G_0^+$. 
By the choice of $r$, contracting all these faces to vertices yields a graph of tree-width at least $t+\alpha$. Removing the new vertices, of which we have at most $\alpha$, results in the graph $G_0\setminus\bigcup\V$ with tree-width at least $t$: 
note that the new edges of $G^+_0$ disappear in these two steps.
Thus, the graph $G_0$ has tree-width at least $t$ as well, proving the lemma. 
\end{proof}

\begin{proof}[Proof of Theorem~\ref{thm:richthm}]
Let $\hat \alpha$ be the integer $\alpha$ provided for $R$ and $m=0$ by Theorem~\ref{thm:extended1.3}, and let $\hat \gamma$ be an integer 
such that $R$ embeds in every surface $\Sigma$ with $g(\Sigma)> \hat \gamma$. By Theorems~\ref{thm:extended1.3} and~\ref{thm:gridthm} and 
Lemmas~\ref{lem:G0haslargetw}, \ref{lem:flatwall} and~\ref{lem:torsotwlarge}, there is an integer $w$ such that 
if the tree-width of our graph~$G$ is larger than $w$, the following holds: 
There is a tree-decomposition $(V_t)_{t\in T}$ of $G$ such that the torso $\hat G$ of one part $V_{t_0}$ 
has an $\hat\alpha$-near embedding $(\hat\sigma,\hat G_0,\hat A, \hat \V',\emptyset)$ in 
a surface $\hat\Sigma$ in which $R$ cannot be embedded such that $\hat G'_0$ contains a flat wall of size at least
\[
	6^{\hat \alpha + 2\hat \gamma +1}(r + \hat\alpha(\beta + \hat \alpha + 3) + p),
\]
where $p:=2\hat\alpha(\beta + 2\hat \alpha + 2\hat \gamma + 4) + 4$\?{$p$}.
We will show that, with these constants, we find an $\alpha$-near embedding of $G$ for $\alpha=(\alpha_0, \alpha_1, \alpha_2)$
defined as\?{$\alpha$}
\begin{align*}
\alpha_0 &:= \hat\alpha+p(2\hat\gamma + \hat \alpha) + 2\hat\alpha^2 + 2\hat \alpha\\
\alpha_1 &:= \hat\alpha + \hat\gamma \\
\alpha_2 &:= 2\hat\alpha + \hat \gamma
\end{align*}
that is almost $(\beta,r)$-rich: The near embedding satisfies all the desired properties except for (\ref{prop:linkagetogrid}) and (\ref{prop:orthogonalpaths}). Instead, we only find paths linking the societies of large vortices to arbitrary branch vertices of a large wall. But this can be remedied: We apply the result for $32\beta^2+r$ and $16\beta^2$ instead for $r$ and $\beta$, respectively, and with Lemmas~\ref{lem:pathstoboundary} and~\ref{lem:ortholink} we obtain a $(\beta,r)$-rich near-embedding as desired.

First, we will convert the near-embedding of the torso $\hat G$ into a near-embedding of 
the whole graph $G$ by accomodating the rest in its vortices. 
To accomplish this, we will use property~(ii) from Theorem~\ref{thm:extended1.3} of our tree-decomposition.
Pick a component $T'$ of $T-t_0$ and let $t'$ be the vertex in this component adjacent to $t_0$ in $T$. 
Let $Y:=(\bigcup_{t\in T'}V_t)\setminus \hat A$. If $t'$ is the parent of $t_0$ in $T$, i.e., if $t'\in rTt_0$,
then $V_{t'}\cap V_{t_0}\subseteq \hat A$ and hence $Y\cap V_{t_0}=\emptyset$. 
We then add $G[Y]$ to $\hat G$ as a small vortex. 
Suppose now that $t'$ is a child of $t_0$, i.e., that $t_0\in rTt'$. 
Then, by (ii) of Theorem~\ref{thm:extended1.3}, we can either add $G[Y]$ to a part $X_{tt'}$ of a vortex $V$ of $\hat G$ without 
increasing the adhesion of $V$, or we can add $G[Y]$ as a small vortex that is properly attached.

We perform this modification for all components of $T-t_0$. 
Let us collect in a set $\hat \W$ the new small vortices defined, and let $\hat\V$ denote the set of the new possibly modified, 
large vortices. By merging vortices if necessary, 
we may assume that there are no two vortices $W,W'\in\hat\W$ with $\Omega(W)\subseteq\Omega(W')$. 
Note that $(\hat\sigma,\hat G_0, \hat A, \hat \V,\hat \W)$ is an $\hat\alpha$-near-embedding of all of~$G$, 
and thus also an $\alpha$-near embedding.

Let us, more generally, consider $\alpha$-near-embeddings $(\sigma,G_0,A,\V,\W)$ of $G$ in surfaces $\Sigma$ such that 
\begin{equation}\label{inequalities}\tag{$\star$}
\left.
\begin{minipage}[c]{0.8\textwidth}\em
\begin{itemize}
\item All vortices in $\W$ are properly attached
\item All vortices in $\V$ have adhesion at most\\ ${\hat \alpha+g(\hat \Sigma)-g(\Sigma)+|\hat\V|-|\V|}$
\item $g(\Sigma)\leq g(\hat \Sigma)$
\item $|\V| \leq |\hat \V| + (g(\hat \Sigma)-g(\Sigma))$ \quad $(\leq \alpha_1)$
\item $|A| \leq |\hat A| +  p\Big(2\big(g(\hat \Sigma) - g(\Sigma)\big) + |\hat \V| - |\V|\Big)$ \quad $(\leq \alpha_0)$
\end{itemize}
\end{minipage}\ignorespacesafterend 
\right\}
\end{equation}
and further
\begin{equation}\label{bigwall}\tag{$\star\,\star$}
\begin{minipage}[c]{0.8\textwidth}\em
$G_0'$ contains a flat wall $H_0$ of size at least $6^q\mu$.
\end{minipage}\ignorespacesafterend 
\end{equation}
where\?{$\mu$, $q$, $\lambda$}
\begin{align*}
	q &:= |\V|+2g(\Sigma)+1 \\
	\mu &:=\mu(\sigma,G_0,A,\V,\W):=r + |\V|(\beta+2\hat\alpha+\hat \gamma+ 3) + p
\end{align*}
Such near-embeddings exist, since $(\hat\sigma,\hat G_0, \hat A, \hat \V,\hat \W)$ satisfies (\ref{inequalities}) and (\ref{bigwall}).

Our next task is to find, among all such near embeddings, one with the following additional properties (P1)--(P4):
\begin{enumerate}[(P1)]
\item Every two vortices have distance at least $2\lambda+3$ in $\Sigma$.
\item For every vortex $V\in\V$ there exist $\lambda$ cycles $(C_1,\ldots,C_\lambda)$ tightly enclosing~$V$. If $\Sigma\not\hom S^2$, the representativity of $G'_0$ in $\Sigma$ is at least $\lambda$.
\item For all distinct vortices $V,W\in\V$, the discs $\closure{D(C_1(V))}$ and $\closure{D(C_1(W))}$ are disjoint.
\item $H_0$ contains a flat wall $H$ of size $6\mu$ such that $D(H)\cap D(C_1(V))=\emptyset$ for every $V\in\V$.
\end{enumerate}
where
\begin{align*}
	\lambda &:= \lambda(\sigma, G_0, A, \V, \W):=|\V|(\beta+2\hat\alpha+\hat\gamma + 3).
\end{align*}
From all $\alpha$-near-embeddings satisfying  (\ref{inequalities}) and (\ref{bigwall}) 
let us pick one minimizing $(g(\Sigma),|\V|)$ lexicographically. 
We will denote this near-embedding by $\e$. 
We will show that either $\e$ itself has the properties (P1)--(P4) or we can find a disc in $\Sigma$ such that, roughly said, the part of our graph nearly-embedded in this disc can be considered as a near-embedding in $S^2$ with these properties.

For the next steps in the proof, we will repeatedly make use of the following fact: 
for integers $\ell, r$, consider a flat wall $W$ of size $8\ell + 2r$ in $G'_0$.
In $W$, we can find two subwalls $W_1,W_2$ of size $r$, together with $\ell$ concentric cycles $C_1(W_1),\ldots,C_\ell(W_1)$ 
around $W_1$ and $\ell$ concentric cycles $C_1(W_2),\ldots,C_\ell(W_2)$ around $W_2$ 
such that $\closure{D(C_1(W_1))}$ and $\closure{D(C_1(W_2))}$ are disjoint. 
In particular, $W_1$ and $W_2$ have distance at least $2\ell+2$ in $\Sigma$. 
Further, if $V$ is a vortex tightly enclosed by $k<\ell$ cycles $C_1(V),\ldots,C_k(V)$, 
then any two vertices picked from the cycles $C_1(V),\ldots,C_k(V)$ have distance at most $2k<2\ell$ in $\Sigma$. 
Now, a comparison of the distances shows that one of the walls $W_1$, $W_2$ is disjoint 
from $\Omega(V)$ and all the cycles $C_1(V),\ldots,C_k(V)$.

Finally note that, if we delete a set $X$ of $k$ vertices from a wall $H$ of size $\ell>k$, at most $k$ rows and 
at most $k$ columns of $H$ are hit by $X$ and thus, $H-X$ contains a wall of size at least $\ell-4k$.

Let us show first that our near-embedding $\e$ has property (P1).  
Otherwise we apply Lemma~\ref{lem:mergevortices} with $d:=2\lambda$. 
This gives a vertex set $A'$ of size at most $2(2\hat\alpha+\hat \gamma)+d\leq p$ and a near-embedding 
$\e':=(\sigma',G_0-A',A\cup A',\V',\W')$ with $|\V'|\leq|\V|-1$ of $G$ in $\Sigma$. 
By Lemma~\ref{lem:properlyattached}, we may assume that its small vortices are properly attached.
Then, $\e'$ satisfies (\ref{inequalities}) and (\ref{bigwall}) but $(g(\Sigma),|\V'|)<(g(\Sigma),|\V|)$ lexicographically, which contradicts the choice of~$\e$. 

To show properties (P2)--(P4) we consider two cases: when $\Sigma \hom S^2$ and when $\Sigma \not\hom S^2$.

First, we assume that $\Sigma \hom S^2$. Given a vortex $V\in\V$ it is easy to find in $H_0$ a subwall $H$ half the size of $H_0$ such that 
$\closure{D(V)} \cap \closure{D(H)}=\emptyset$.
Then $H$ has size at least $3\cdot 6^{q-1}\mu > 6^{q-1}\mu + 4\lambda$
(as $q-1 \geq |\V|\geq 1$).
Inside $H$ there is a flat wall $H_V$ of size at least $6^{q-1}\mu$ enclosed by $\lambda$ cycles $C_1,\ldots,C_\lambda\subseteq H$.
As $\Sigma \hom S^2$, $C_\lambda,\ldots,C_1$ enclose $V$, 
and $\closure{D(C_\lambda)} \cap \closure{D(H_V)} = \emptyset$. 
Lemma~\ref{lem:tightcycles} shows that there also exist $\lambda$ cycles  $C_1(V),\ldots,C_\lambda(V)$ tightly enclosing~$V$ 
such that $D(C_1(V))$ does not meet $D(H_V)$. Iterating this procedure for all $V\in\V$ establishes (P2), 
while replacing our original wall $H_0$ with a flat subwall $H$ of size at least $6^{q-|\V|}\mu\geq 6\mu$ that
satisfies $\closure{D(H)}\cap\closure{D(C_1(V))}=\emptyset \ \forall\  V\in\V$. 

To prove (P3) suppose, that for two vortices $V,W\in\V$ the discs $\closure{D(C_1(V))}$ and $\closure{D(C_1(W))}$ intersect.
By (P1), all the cycles $C_1(V),\ldots,C_\lambda(V),C_1(W),\ldots,C_\lambda(W)$ are disjoint, 
so we may assume that $D(C_1(V))\subseteq D(C_1(W))$. 
By Lemma~\ref{lem:pieslices}, there is a disc $\Delta\subseteq D(C_1(W))$ containing $D(V)$ and a separation $(X_1,X_2)$ 
of $G$ of order at most $2\lambda+2\hat\alpha\leq p$ such that $G_0\cap X_1 = G_0\cap \Delta$. 
Now let $\tilde \V$ be the set of all vortices of $\V-X_1$ with a society of size at least $4$, 
and let $\tilde \W$ be the set of all vortices of $\W-X_1$, the vortex $(G[X_1],\emptyset)$, 
and the vortices of $\V-X_1$ with a society of at most $3$ vertices.
Clearly, $|\tilde \V|<|\V|$, as $\Omega(V)\subseteq X_1$. It is easy to see now that 
\[
	(\sigma|_{G_0-X_1}, G_0-X_1,A\cup (X_1\cap X_2),\tilde \V,\tilde \W)
\]
is an $(\alpha_0,\alpha_1,\alpha_2)$ near-embedding of $G$ in $\Sigma$, satisfying (\ref{inequalities}), 
and as $H$ is a sufficiently large wall living in $G_0-X_1$, condition (\ref{bigwall}) holds as well. 
This means that $(g(\Sigma),|\tilde \V|)<(g(\Sigma),|\V|)$, a contradiction to our choice of $\e$. This proves (P3) and (P4). 

We now consider the case when $\Sigma \not \hom S^2$. Our plan is to deduce (P2) from (P1) and Lemmas~\ref{lem:conccycles} and Lemmas~\ref{lem:rephigh}.
Thus, we must first show that the representativity of $G'_0$ in $\Sigma$ is at least $2\lambda+2$. 
Suppose not, and apply Lemma~\ref{lem:rephigh} with $z:=2\lambda+2$. 
If (a) of Lemma~\ref{lem:rephigh} holds, we have a near-embedding $\e'$ of $G$ in a surface $\Sigma'$ 
with $g(\Sigma')<g(\Sigma)$ and $|\V'|\leq|\V|+1$. 
The properties (\ref{inequalities}) and (\ref{bigwall}) are easy to verify;
for (\ref{bigwall}), notice that $5\cdot 6^{q-1}\mu \geq 8\lambda + 8$, so deleting up to $z$ vertices from our wall $H$ leaves a wall
of size at least $6^{q-1}\mu$. 
Hence, the fact that $(g(\Sigma'),|\V'|)<(g(\Sigma),|\V|)$ contradicts our choice of $\e$. 
Similarly, if (b) of Lemma~\ref{lem:rephigh} holds, then one of the graphs ${G'}_0^1, {G'}_0^2$ contains a sufficiently large wall,
so one of the near-embeddings $\e_1,\e_2$ satisfies (\ref{inequalities}) and (\ref{bigwall}), 
and $g(\Sigma_1),g(\Sigma_2)<g(\Sigma)$ yields the same contradiction as before. 
This shows that the representativity of $G'_0$ in $\Sigma$ is at least $2\lambda+2$.
We now apply Lemma~\ref{lem:conccycles} to each of the faces of $G'_0$ that contain 
the disc $D(V)$ of a vortex $V\in\V$. Together with Lemma~\ref{lem:tightcycles}, this implies property (P2).

To show property (P3), assume that for two vortices $V,W\in\V$ the discs $\closure{D(C_1(V))}$ and $\closure{D(C_1(W))}$ intersect. 
As before we may assume that $D(C_1(V))\subseteq D(C_1(W))$,
and an application of Lemma~\ref{lem:pieslices} gives us a disc $\Delta\subseteq D(C_1(W))$ containing $D(V)$, and 
a separation $(X_1,X_2)$ of $G$ of order at most $2\lambda+2\hat\alpha\leq p$ such that $G_0\cap X_1 = G_0\cap \Delta$. 
As noted earlier, there exists a flat subwall $H$ of $H_0$ of size at least $6^{q-1}\mu$ that is disjoint from $\Omega(W)$ and 
all the cycles $(C_1(W),\ldots,C_\lambda(W))$, and hence from $X_1\cap X_2$ 
(as $X_1\cap X_2 \cap G'_0 \subseteq \bigcup C_i(W)\cup \Omega(W)$ in Lemma~\ref{lem:pieslices}). 
If $D(H)\cap \Delta = \emptyset$, then $H\subseteq G[X_2\setminus X_1]$, so turning $G[X_1]$ into a small vortex
attached (by an empty society) to $G'_0\cap G[X_2\setminus X_1]$
we can reduce the number of large vortices of our near-embedding, leading to the same contradiction as earlier. 
Otherwise, $D(H)\subseteq \Delta$ and $H\subseteq G[X_1\setminus X_2]$. We now turn $G[X_2]$ into a small vortex attached 
to $G'_0\cap G[X_1\setminus X_2]$ and obtain a contradiction to the minimality of $\e$, since $g(\Delta)=0<g(\Sigma)$.
This completes the proof of (P3) for the case of $\Sigma \not \hom S^2$.

To show (P4), let us enumerate the vortices $\V=:\{V_1,\ldots,V_\ell\}$. 
We will prove by induction on $k$ that, for $1\leq k\leq\ell$, there is a flat wall $H_k \subseteq H_{k-1}$ of size $6^{q-k}\mu$
such that $D(H_k)$ avoids $D(C_1(V_1)),\ldots,D(C_1(V_k))$. 
For $k=0$ this is precisely (\ref{bigwall}). Given $k\geq 1$, we have $q\geq k+3$ and $\lambda\leq \mu/2$. 
Hence, as earlier, we can find a subwall $H_k \subseteq H_{k-1}$ of size $6^{q-(k-1)}\mu/6\geq 6^{q-k}\mu \geq 6^3\mu$ 
from (P2) that avoids $\Omega(V_k)$ and all cycles $C_1(V_k),\ldots,C_\lambda(V_k)$.
If $D(H_k)\cap D(C_1(V_k)) = \emptyset$, then this completes the induction step. 
Otherwise, Lemma~\ref{lem:pieslices} gives us a disc $\Delta\subseteq D(C_1(V_k))$ containing $D(H_k)$ 
and a separation $(X_1,X_2)$ of $G$ of order at most $2\lambda+2\hat\alpha\leq p$ such that $G_0\cap X_1 = G_0\cap \Delta$
and $X_1\cap X_2\cap V(G_0)\subseteq \bigcup C_i (V_k) \cup \Omega(V_k)$.
By (P3), this disc $\Delta$ does not contain $D(W)$ for any large vortex $W\neq V_k$.
We now turn $G[X_2\setminus X_1]$ into a small vortex attached by an empty society to $G'_0\cap G[X_1\setminus X_2]$. 
We are now back in the case of $\Sigma = S^2$ treated before, and can find inside $H_k$
(which we recall has size at least $6^{q-k}\geq 6^3\mu \geq 6\mu + 4\lambda$) 
a flat subwall $H'$ of size $6\mu+4\lambda$ with $\closure{D(H')}\cap \closure{D(V_k)}=\emptyset$.
Inside $H'$ there is a wall of size $6\mu$ (which we note is large enough to satisfy (\ref{bigwall}) for our large vortex and genus $0$)
enclosed by $\lambda$ cycles in $H'$. Re-interpreting these cycles as enclosing $V_k$, as earlier in our proof of (P2)--(P4) for $\Sigma\hom S^2$,
we once more obtain a contradiction to the minimality of $\e$. 
This completes the proof of (P4) for the case of $\Sigma\not \hom S^2$.%
   \COMMENT{}

From all $\alpha$-near-embeddings $(\sigma,G_0,A,\V,\W)$ of $G$ into surfaces $\Sigma$ satisfying (\ref{inequalities}) and (\ref{bigwall}) 
and (P1)--(P4) let us choose one minimizing $|\V|$. Let $H$ be the wall from (P4).

An application of Lemma~\ref{lem:vorticeslinked} now gives us a subgraph $\tilde G_0$ of $G_0$, 
a vertex set $\tilde A\subseteq V(G)\setminus V(\tilde G_0)$ with $|\tilde A|\leq 2\hat \alpha^2 + 2\hat \alpha$ disjoint from $H$ and 
an $\alpha$-near-embedding $\tilde \e := (\tilde\sigma, \tilde G_0, A\cup \tilde A, \tilde\V,\tilde\W)$ of $G$ 
such that every vortex $\tilde V\in\tilde\V$ has a linked decomposition of adhesion at most $\hat \alpha$,
there are still at least $\tilde\lambda:=\lambda-(\hat \alpha+1)$ cycles enclosing every $\tilde V\in\tilde\V$.


Let us show that there is no vertex set $S$ in $G'_0$ of size less than $\beta$ separating $H$ from $\Omega(V)$ for some vortex $V\in \tilde \V$. 
Suppose there is and
let us choose $S$ minimal.
We add to $\tilde G'_0$ a vertex $v$ and edges from $v$ to all society vertices $\Omega(V)$. 
Clearly, after adding $v$ to $X_2$, the set $S$ still separates $H$ from $\Omega(V)$ and 
we can extend our embedding by mapping $v$ and the new edges to $D(V)$. 
By the minimality of $S$, every vertex in $S$ is adjacent to some vertex of the component
$T_0$ of $G'_0-S$ that contains $v$. Let $T$ denote the (connected) graph $T_0$ together with $S$ and all edges between $T_0$ and $S$.
We note that $T_0$ contains $\Omega(V)$.

We fatten the embedded graph $T$ to obtain a closed, connected set $D\subseteq \Sigma$ 
so that $D$ contains $T$ and 
further that that $\boundary D$ intersects with $G'_0$ only in edges incident with both $S$ and $G'_0 \setminus T$. 

Every component $C$ of $\boundary D$ bounds a cycle in $\Sigma$. 
This is clear if $\Sigma\hom S^2$ and if $\Sigma \not \hom S^2$
we could otherwise slightly shift $C$ in neighbourhoods of vertices in $S$ 
to intersect with $G'_0$ only in $S$ and obtain a genus reducing curve that hits $G'_0$ ￼
in less than $\beta$ many vertices, contradicting~(P2).

Let $H'$ be a subwall of $H$ of size at least $6\mu-4\beta$ that avoids $S$ and 
let $Z$ be the component of $\Sigma-\boundary D$ containing $H'$. We define $X_1:=(V(G'_0)\cap Z)\cup S$
and $X_2:=V(G'_0)\cap (\Sigma\setminus Z)$. This gives us a separation $(X_1,X_2)$ with $S\subseteq X_1\cap X_2$.

Further $(X_1, X_2)$ can be modified so that, for every vortex $V'\in\tilde\V\setminus\{V\}$, its society $\Omega(V')$ 
and at least $\tilde\lambda-\beta$ many cycles enclosing $V'$ are contained either in $X_1$ or in $X_2$. 
Indeed, as one of the cycles $C_1(V'),\ldots,C_\beta(V')$ enclosing $V'$ is not hit by $X_1\cap X_2$, it is contained in either $X_1$ or $X_2$. 
As this (plane) cycle $C$ is a separator of $G'_0$, we can add all the vertices embedded in $D(C)$,
in particular $\Omega(V')$ and the vertices of the cycles $C_{\beta+1},\ldots, C_{\tilde\lambda}$, to the same $X_i$.

Let us consider the case that $Z$ is a disc. As before, we add $X_1\cap X_2$ to the apex set $A\cup \tilde A$ 
and add a new small vortex $W:=(G[X_1\setminus X_2],\emptyset)$ to $\tilde\W$. 
This new near-embedding of $G$ in $S^2$ still satisfies (\ref{inequalities}), (\ref{bigwall}), and (P1)--(P4), 
but we have reduced $|\tilde\V|$ with this operation, a contradiction to our choice of the near-embedding. 

Suppose $Z$ is not a disc. 
Then, as each component of $\boundary Z$ bounds a disc and 
each component of $\Sigma \setminus Z$ contains a point of $D$, 
which is connected, $\Sigma\setminus Z$ has exactly one component $Z'\hom S^2$. 
Again, we add $X_1\cap X_2$ to the apex set and accomodate $X_2\setminus X_1$ in a 
small vortex and obtain a near embedding of $G$ with a reduced number of large vortices. 
Let us show that the representativity of our new embedded graph $\hat G:= \tilde G'_0-X_2$ 
is at least $\lambda-\beta$. 
Assume the opposite and pick a genus reducing $\hat G$-normal curve $C$ in $\Sigma$ that 
meets less than $\lambda-\beta$ vertices of $\hat G$. 
Then, we may assume by Lemma~\ref{lem:curvehitsvortexonce} that $C$ intersects the 
face $f$ of $\hat G$ containing $Z'$ at most once. 
We reroute $C$ in $f$ along $\boundary Z'$. 
Now, $C$ is also a $\tilde G'_0$-normal curve 
meeting at most $\beta$ many additional vertices from $S$, 
which contradicts (P2) for $\tilde G'_0$. 
Now, (P1)--(P4) are easy to verify.

We conclude that, for every large vortex $V\in\tilde\V$, the society $\Omega(V)$ is connected to the branch vertices of $H$ by $\beta$ many disjoint paths. 


Finally, as described in the beginning, Lemmas~\ref{lem:pathstoboundary} and~\ref{lem:ortholink} finish the proof.
\end{proof}


\section{Circular Vortex-Decompositions}
\label{sec:circulardecompositions}


In Graph Minors XVII \cite{GM17}, the structure theorem is stated with vortices having a circular instead of a linear structure. 
For most applications, the linear decompositions as discussed so far in this paper are sufficient, but sometimes the circular structure is necessary. 
In this section, we introduce circular vortex decompositions and point out how we can derive a new lemma from the proof of Lemma~\ref{lem:vorticeslinked} that yields circular linkages for them. It is easy to see that we can apply this new lemma instead of Lemma~\ref{lem:vorticeslinked} at the end of the proof of Theorem~\ref{thm:richthm} and therefore, we can choose to have circular linkages for the large vortices when we apply the theorem.

For the remainder of this paper, we call decompositions of vortices as defined in Section~\ref{sec:structure_thms} \emph{linear decompositions} to distinguish them more clearly from the circular decompositions which we introduce now:

Let $V:=(G,\Omega)$ be a vortex with $\Omega=(w_1,\ldots,w_n)$. Let us regard the ordering of $\Omega$ as a cyclic ordering. A tuple $\D:=(X_1,\ldots,X_n)$ of subsets of~$V(G)$ is a \emph{circular decomposition} of $V$ if the following properties are satisfied:
\begin{enumerate}[(i)]
\item $w_i\in X_i$ for all $1\leq i \leq n$.
\item $X_1\cup\ldots\cup X_n=V(G)$.
\item When $w_i<w_j<w_k<w_\ell$ are society vertices of $V$ ordered with respect to the cyclic ordering $\Omega$, then $X_i\cap X_k\subseteq X_j\cup X_\ell$
\item Every edge of $G$ has both ends in $X_i$ for some $1\leq i \leq n$.
\end{enumerate}

The \emph{adhesion} of our circular decomposition $\D$ of $V$ is the maximum value of ${|X_{i-1}\cap X_{i}|}$, taken over all $1\leq i\leq n$. We define the \textit{circular adhesion} of $V$ as the minimum adhesion of a circular decomposition of that vortex.

When $\D$ is a circular decomposition of a vortex $V$ as above, we write ${Z_i:=(X_{i}\cap X_{i+1})\setminus \Omega}$, for all $1 \leq i< n$.  These $Z_i$ are the \emph{adhesion sets} of~$\D$. We call $\D$ \textit{linked} if 
\begin{itemize}
\item all these $Z_i$ have the same size;
\item there are $|Z_i|$ disjoint $Z_{i-1}$--$Z_{i}$ paths in~$G[X_i]-\Omega$, for all $1\leq i\leq n$;
\item $X_i \cap \Omega =\{w_{i-1},w_{i}\}$ for all $1\leq i \leq n$.
\end{itemize}
Note that $X_i\cap X_{i+1} = Z_i \cup \{w_i\}$, for all $1\leq i \leq n$ (Fig.~\ref{fig:linkedvortex}).

The union of the $Z_{i-1}$--$Z_{i}$ paths in a circular decomposition of~$V$ is a disjoint union of cycles in~$G$ each of which traverses the adhesion sets of~$\D$ in cyclic order (possibly several times); We call the set of these cycles a \emph{circular linkage} of~$V$ with respect to~$\D$.

As described in Section~\ref{sec:structure_thms} for linear decompositions, we see that we can delete a vertex from a circular decomposition of some vortex and obtain a new circular decomposition. This operation does not increase the adhesion but might decrease the number of society vertices.

Clearly, a linear decomposition of some vortex is a circular decomposition as well and it is easy to see that one can obtain a linear from a circular decomposition, if one deletes the overlap of two subsequent bags: Let $V:=(G,\Omega)$ a vortex and $(X_1,\ldots,X_n)$ a circular decomposition of $V$.
Delete the set $X_{i-1}\cap X_i$ from $V$ for some index $1\leq i\leq n$. We obtain a circular decomposition $\D:=(X'_1,\ldots,X'_{n'})$ of  $V-Z$ with $n'\leq n$. By shifting the indices if necessary we may assume that $X'_{n'}\cap X'_1$ is empty. $\D$ is a linear decomposition of $V-Z$: Pick a vertex $v\in X'_j\cap X'_\ell$ for indices $1\leq j <\ell\leq n'$. This vertex avoids either $X'_1$ or $X'_{n'}$, let us assume the former. We apply property (iii) from the definition of a circular decomposition to $w_1,w_j,w_k,w_\ell$ for any $k$ with $j < k < \ell$ and conclude that $v\in X'_k$.

To distinguish near-embeddings with linear decompositions from near-em\-bed\-dings with circular decompositions, we will call the latter explicitly \emph{near-embeddings with circular vortices}. Also, for a $(\alpha_0,\alpha_1,\alpha_2)$-near embedding with circular vortices let the third bound $\alpha_2$ denote an upper bound for the circular adhesion of the large vortices.

We give a modified definition of $\beta$-rich to comply with the new concepts. For near-embeddings with circular decompositions we replace property~(\ref{prop:linked}) by the following:
\begin{enumerate}
\item[(\ref{prop:linked}')]
Let $V \in \V$ with $\Omega(V)=(w_1,\ldots,w_n)$. Then there is a circular, linked decomposition of $V$ of adhesion at most $\alpha_2$ and a cycle $C$ in $V\cup \bigcup \W$ with $V(C\cap G_0)=\Omega(V)$ that avoids all the cycles of the circular linkage of~$V$, and traverses $w_1,\ldots,w_n$ in this order.
\end{enumerate}

\begin{lem}\label{lem:circularvorticeslinked}
Let $(\sigma,G_0,A,\V,\W)$ be  an $(\alpha_0,\alpha_1,\alpha_2)$-near embedding of a graph $G$ in a surface $\Sigma$ 
such that every small vortex $W\in\W$ is properly attached.
Moreover, assume that
\begin{enumerate}[\rm (i)]
\item 
For every vortex $V\in \V$ there are $\alpha_2+1$ concentric cycles $C_0(V),\ldots,C_{\alpha_2}(V)$ in $G'_0$ tightly enclosing $V$.
\item For distinct vortices $V,W\in\V$, the discs $\closure{D(C_0(V))}$ and $\closure{D(C_0(W))}$ are disjoint.
\end{enumerate}
Then there is a graph $\tilde G_0\subseteq G_0$ containing $G_0\setminus \Big (\bigcup_{V\in\V}D(C_0(V))\Big)$, 
a set $\tilde A\subseteq V(G)\setminus V(\tilde G_0)$ of size $|\tilde A|\leq\tilde\alpha:= \alpha_0+ \alpha_1(2\alpha_2+2)$, 
and sets $\tilde\V$ and $\tilde\W\subseteq \W$ of vortices
such that, with $\tilde \sigma:= \sigma|_{\tilde G'_0}$, the tuple $(\tilde\sigma,\tilde G_0,A\cup\tilde A,\tilde\V,\tilde\W)$ 
is an  $(\tilde\alpha,\alpha_1,\alpha_2+1)$-near embedding with circular vortices of $G$ in $\Sigma$ such that every vortex 
$\tilde V\in \tilde\V$ satisfies condition {\rm (\ref{prop:linked}')} of the definition of $(\beta, r)$-rich,
and $D(\tilde V)\supseteq D(V)$ for some $V\in\V$.
\end{lem}
\begin{proof}
This lemma can be proven almost exactly like Lemma~\ref{lem:vorticeslinked}. To avoid completely rewriting the proof, we just point out the differences. 

The curve $C$ in the surface hits the vertex set $S$ which consists of exactly one vertex from each $C_i(V)$ and one society vertex $w'_j$ of $V$. We split each vertex in $S\setminus \{w'_j\}$: For each  $0\leq i \leq \alpha_2$, we replace $v\in S\cap V(C_i(V))$ by two new vertices $x_i, y_i$ and connect them with edges to the former neighbours of $v$ such that $C$ does not intersect any edges or vertices. The vertices $x_0,\ldots,x_{\alpha_2}$ and $y_0,\ldots,y_{\alpha_2}$ form the sets $X$ and $Y$, respectively. 

In the remainder of the proof we delete the set $Z$ instead of $S\cup Z$. At the end, we identify the vertex pairs $(x_i,y_i)$ for $0\leq i \leq \alpha_2$ and obtain a linked, circular decomposition as desired.
\end{proof}


\newpage

\small
\begin{tabular}{cc}

\begin{minipage}[t]{0.5\linewidth}
Reinhard Diestel\\
Mathematisches Seminar\\
Universit\"at Hamburg\\
Bundesstra\ss e 55\\
20146 Hamburg\\
Germany
\end{minipage}
&
\begin{minipage}[t]{0.5\linewidth}
Ken-ichi Kawarabayashi\\
National Institute of Informatics,\\
2-1-2, Hitotsubashi, Chiyoda-ku, Tokyo, Japan.\\
{\tt <k\_keniti@nii.ac.jp>}\\
{\footnotesize Research partly
supported by Japan Society for the
Promotion of Science, Grant-in-Aid for Scientific Research,
by C \& C Foundation, by
Kayamori Foundation and by Inoue Research Award
for Young Scientists.}
\end{minipage}
\\
~\vspace{5mm}
\\
\begin{minipage}[t]{0.5\linewidth}
Theodor M\"uller (corresponding author)\\
Dept.\ Mathematik\\
Universit\"at Hamburg\\
Bundesstra\ss e 55\\
20146 Hamburg\\
Germany\\
{\tt <mathematik@theome.com>}\\
{\footnotesize Research partially supported by NII MOU grant.}
\end{minipage}
&
\begin{minipage}[t]{0.5\linewidth}
Paul Wollan\\
Department of Computer Science\\
University of Rome ``La Sapienza"\\
Via Salaria 113\\
Rome, 00198 Italy\\
{\footnotesize Research partially supported by a research fellowship from the Alexander von Humboldt Foundation at the University of Hamburg and by NII MOU grant.}
\end{minipage}

\end{tabular}

\end{document}